\documentclass[11pt,reqno]{article}
\usepackage[utf8]{inputenc}
\usepackage{geometry}
\usepackage{color}
\usepackage{amsmath}
\usepackage{amsthm}
\usepackage{amssymb}
\usepackage{graphicx}
\usepackage[all]{xy}
\PassOptionsToPackage{normalem}{ulem}
\usepackage{ulem}
\usepackage[unicode=true,pdfusetitle,
 bookmarks=true,bookmarksnumbered=false,bookmarksopen=false,
 breaklinks=false,pdfborder={0 0 1},backref=page,colorlinks=true]
 {hyperref}
\hypersetup{
 linkcolor=blue,citecolor=blue}

\makeatletter
\numberwithin{equation}{section}
\numberwithin{figure}{section}
\theoremstyle{plain}
\newtheorem{thm}{\protect\theoremname}[section]
\theoremstyle{definition}
\newtheorem{defn}[thm]{\protect\definitionname}
\theoremstyle{remark}
\newtheorem{rem}[thm]{\protect\remarkname}
\theoremstyle{plain}
\newtheorem{prop}[thm]{\protect\propositionname}
\theoremstyle{definition}
\newtheorem{example}[thm]{\protect\examplename}
\theoremstyle{plain}
\newtheorem{lem}[thm]{\protect\lemmaname}
\theoremstyle{plain}
\newtheorem{cor}[thm]{\protect\corollaryname}

\usepackage{graphicx}
\usepackage{appendix}

\usepackage{enumitem}
\usepackage[backref=page]{hyperref}

\date{\today}

\usepackage{ytableau}

\makeatother

\providecommand{\corollaryname}{Corollary}
\providecommand{\definitionname}{Definition}
\providecommand{\examplename}{Example}
\providecommand{\lemmaname}{Lemma}
\providecommand{\propositionname}{Proposition}
\providecommand{\remarkname}{Remark}
\providecommand{\theoremname}{Theorem}

\begin{document}
\global\long\def\F{\mathrm{\mathbf{F}} }%
\global\long\def\Aut{\mathrm{Aut}}%
\global\long\def\C{\mathbf{C}}%
\global\long\def\H{\mathcal{H}}%
 
\global\long\def\U{\mathcal{U}}%
 
\global\long\def\U{\mathcal{U}}%
 
\global\long\def\lab{\ell}%
 
\global\long\def\tsg{\widetilde{\Sigma_{g}}}%

\global\long\def\ext{\mathrm{ext}}%
 
\global\long\def\triv{\mathrm{triv}}%
 
\global\long\def\Hom{\mathrm{Hom}}%

\global\long\def\trace{\mathrm{tr}}%
 
\global\long\def\rk{\mathrm{rk}}%

\global\long\def\L{\mathcal{L}}%
\global\long\def\W{\mathcal{W}}%
\global\long\def\SL{\mathrm{SL}}%
 
\global\long\def\A{\mathcal{A}}%
\global\long\def\a{\mathbf{a}}%

\global\long\def\D{\mathcal{D}}%
\global\long\def\df{\mathrm{def}}%
\global\long\def\eqdf{\stackrel{\df}{=}}%
\global\long\def\Tr{\mathrm{Tr}}%
\global\long\def\std{\mathrm{std}}%
 
\global\long\def\HS{\mathrm{H.S.}}%
\global\long\def\e{\mathbf{e}}%
\global\long\def\c{\mathbf{c}}%
\global\long\def\d{\mathbf{d}}%
\global\long\def\AA{\mathbf{A}}%
\global\long\def\BB{\mathbf{B}}%
\global\long\def\u{\mathbf{u}}%
  
\global\long\def\v{\mathbf{v}}%
\global\long\def\spec{\mathrm{spec}}%
\global\long\def\Ind{\mathrm{Ind}}%
\global\long\def\half{\frac{1}{2}}%
\global\long\def\Re{\mathrm{Re}}%
 
\global\long\def\Im{\mathrm{Im}}%
\global\long\def\Rect{\mathrm{Rect}}%
\global\long\def\Crit{\mathrm{Crit}}%
\global\long\def\Stab{\mathrm{Stab}}%
\global\long\def\SL{\mathrm{SL}}%
\global\long\def\Tab{\mathrm{Tab}}%
\global\long\def\Cont{\mathrm{Cont}}%
\global\long\def\I{\mathcal{I}}%
\global\long\def\J{\mathcal{J}}%
\global\long\def\short{\mathrm{short}}%
\global\long\def\Id{\mathrm{Id}}%
\global\long\def\B{\mathcal{B}}%
\global\long\def\ax{\mathrm{ax}}%
\global\long\def\cox{\mathrm{cox}}%
\global\long\def\row{\mathrm{row}}%
\global\long\def\col{\mathrm{col}}%
\global\long\def\X{\mathbb{X}}%
\global\long\def\Fat{\mathsf{Fat}}%

\global\long\def\V{\mathcal{V}}%
\global\long\def\P{\mathbb{P}}%
\global\long\def\Fill{\mathsf{Fill}}%
\global\long\def\fix{\mathsf{fix}}%
 
\global\long\def\reg{\mathrm{reg}}%
\global\long\def\edge{E}%
\global\long\def\id{\mathrm{id}}%
\global\long\def\emb{\mathrm{emb}}%

\global\long\def\Hom{\mathrm{Hom}}%
 
\global\long\def\F{\mathrm{\mathbf{F}} }%
  
\global\long\def\pr{\mathrm{Prob} }%
 
\global\long\def\tr{{\cal T}r }%
\global\long\def\core{\mathrm{Core}}%
\global\long\def\pcore{\mathrm{PCore}}%
\global\long\def\im{\vartheta}%
\global\long\def\br{\mathsf{BR}}%
 
\global\long\def\sbr{\mathsf{SBR}}%
 
\global\long\def\ebs{\mathsf{EBS}}%
 
\global\long\def\ev{\mathrm{ev}}%
 
\global\long\def\CC{\mathcal{C}}%
 
\global\long\def\PP{\mathcal{P}}%
 
\global\long\def\sides{\mathrm{Sides}}%
\global\long\def\tp{\mathrm{top}}%
\global\long\def\lf{\mathrm{left}}%
\global\long\def\MCG{\mathrm{MCG}}%
\global\long\def\EE{\mathcal{E}}%
 
\global\long\def\mog{\mathfrak{MOG}}%
 
\global\long\def\fg{\le_{\mathrm{f.g.}}}%
 
\global\long\def\v{\mathfrak{v}}%
\global\long\def\e{\mathfrak{e}}%
\global\long\def\f{\mathfrak{f}}%
 
\global\long\def\d{\mathfrak{d}}%
\global\long\def\he{\mathfrak{he}}%

\global\long\def\defect{\mathrm{Defect}}%
\global\long\def\M{\mathcal{M}}%
\global\long\def\sdefect{\max\defect}%

\title{Core Surfaces}
\author{Michael Magee and Doron Puder}
\maketitle
\begin{abstract}
Let $\Gamma_{g}$ be the fundamental group of a closed connected orientable
surface of genus $g\geq2$. We introduce a combinatorial structure
of \emph{core surfaces,} that represent subgroups of $\Gamma_{g}$.
These structures are (usually) 2-dimensional complexes, made up of
vertices, labeled oriented edges, and $4g$-gons. They are compact
whenever the corresponding subgroup is finitely generated. The theory
of core surfaces that we initiate here is analogous to the influential
and fruitful theory of Stallings core graphs for subgroups of free
groups.
\end{abstract}
\tableofcontents{}

\section{Introduction\label{sec:Introduction}}

In his influential paper \cite{stallings1983topology}, Stallings
introduced the simple yet powerful concept of \emph{core graphs},
sometimes known today under the name \emph{Stallings core graphs}.
Roughly, core graphs are connected, directed and edge-labeled graphs
in one-to-one correspondence with the (conjugacy classes of) subgroups
of a given f.g.~(finitely generated) free group. We give the exact
definition in Section \ref{sec:core-graphs} below. Core graphs are
especially useful when the corresponding subgroup is f.g., or, equivalently,
when the core graph is finite.

Inter alia, core graphs can be used to extract basic information about
the subgroup (index, rank) (for these and some of the applications
below consult \cite{stallings1983topology} and the surveys \cite{kapovich2002stallings,delgado2022list}).
They provide simple proofs to classical theorems, such as Howson's
theorem that the intersection of two f.g.~subgroups is f.g., Hall's
theorem that every f.g.~subgroup is a free factor in a finite index
subgroup, or Takahasi theorem that given a f.g.~subgroup $H$ of
the free group $\F$, every supergroup $H\le J\le\F$ is a free extension
of one of finitely many supergroups of $H$, to name a few. Core graphs
also give rise to algorithms for various natural problems: for instance,
determine the subgroup generated by a given set of words and the
membership of other words in it, or determine whether a given word
is primitive (a basis element) in a given subgroup. Finally, core
graphs take part in the proofs of more involved results such as in
\cite{PP15}.

In the current paper we wish to define an analogous notion, we call
\emph{core surfaces}, when a free group and its subgroups are replaced
by a surface group and its subgroups. Here, a surface group is the
fundamental group of $\Sigma_{g}$, a closed connected orientable
surface of genus $g\geq2$. We denote this group by $\Gamma_{g}$:
\begin{equation}
\Gamma_{g}\eqdf\pi_{1}\left(\Sigma_{g}\right)\cong\left\langle a_{1},b_{1},\ldots,a_{g},b_{g}\,\middle|\,\left[a_{1},b_{1}\right]\cdots\left[a_{g},b_{g}\right]\right\rangle .\label{eq:Gamma_g}
\end{equation}
In order to motivate our definition of a core surface, we first recall
one of the (equivalent) definitions of a core graph. Let $B_{r}$
be a bouquet consisting of a single vertex and $r$ petals, namely,
a wedge of $r$ copies of $S^{1}$. Denote the wedge point by $o$.
 We identify the fundamental group $\pi_{1}\left(B_{r},o\right)$
with $\F_{r}$, the free group of rank $r$.  Given a subgroup $\left\{ 1\right\} \ne H\le\F_{r}$,
consider the connected covering space $p\colon\Upsilon\to B_{r}$
corresponding to the conjugacy class of\footnote{Namely, $\Upsilon$ is the unique connected covering space such that
for some (and therefore every) vertex $v$ of $\Upsilon$, the image
$p_{*}\left(\pi_{1}\left(\Upsilon,v\right)\right)$ in $\pi_{1}\left(B_{r},o\right)$
is conjugate to $H$.} $H$. The core graph of $H$ is then the subgraph of $\Upsilon$
which is the union of all non-backtracking cycles in $\Upsilon$,
together with the restriction of the covering map $p$. In other words,
we remove from the covering space $\Upsilon$ all the ``hanging trees'',
which do not affect its fundamental group. Equivalently, this is the
unique smallest deformation retract of $\Upsilon$. A key advantage
of core graphs over the original covering spaces is that whenever
$H$ is $f.g.$ but not of finite index, the covering space of $H$
is an infinite graph, while the core graph is a finite one. 

Now let $\Gamma_{g}$ be as in (\ref{eq:Gamma_g}). Given a subgroup
$J\le\Gamma_{g}$, consider the covering space $\Upsilon$ of $\Sigma_{g}$
corresponding to the conjugacy class of $J$. We would like to take
a ``topological core'' of $\Upsilon$. Naturally, when $J$ is of
finite index, the covering $\Upsilon$ is a closed compact surface
and it makes sense to take it as the core surface of $J$. But what
is the appropriate definition when $J$ has infinite index in $\Gamma_{g}$? 

In particular, consider the case where $J$ is f.g.~but of infinite
index (in particular, $J$ is a f.g.~free group). Let us go through
some possible definitions of a core surface which we do \emph{not}
find appealing. 

\paragraph{Smallest retract}

Defining the core surface as a minimal retract would not work: $\Upsilon$
admits a minimal retract which is a finite graph, but this graph is
far from canonical. 

\paragraph{Geodesic boundary}

Another option is to ``trim'' pieces of $\Upsilon$ which, like
hanging trees, do not affect the homotopy type of $\Upsilon$. The
best analog of hanging trees in $\Upsilon$ are ``funnels'': non-compact
pieces that can be cut from $\Upsilon$ by a simple closed curve and
are then homeomorphic to a once-punctured disc. Funnels make $\Upsilon$
non-compact even when $J$ is f.g., and cutting outside a certain
simple closed curve around every funnel leaves us with a compact retract
of $\Upsilon$. The question is, though, which curve should be used
for that. One possibility is to give $\Sigma_{g}$ a hyperbolic structure
(so a Riemannian geometry with constant curvature $-1$). This geometry
can be pulled back to the cover $\Upsilon$. In hyperbolic surfaces,
the homotopy class of every closed curve has a unique geodesic representative,
and one can cut along the unique geodesic representing the simple
closed curve around every funnel of $\Upsilon$. Although this definition
is natural and appealing, the construction is not combinatorial and
therefore loses many of the flexibility we have with Stallings core
graphs.

\medskip{}

The definition we do give is a combinatorial construction which is
close in spirit to the ``geodesic boundary'' definition. Consider
the construction of $\Sigma_{g}$ from a $4g$-gon by identifying
its edges in pairs according to the pattern $a_{1}b_{1}a_{1}^{-1}b_{1}^{-1}\ldots a_{g}b_{g}a_{g}^{-1}b_{g}^{-1}$.
This gives rise to a CW-structure on $\Sigma_{g}$ consisting of one
vertex (denoted $o$), $2g$ oriented $1-$cells (denoted $a_{1},b_{1},\ldots,a_{g},b_{g}$)
and one $2$-cell which is the $4g$-gon glued along $4g$ $1$-cells\footnote{We use the terms vertices and edges interchangeably with $0$-cells
and $1$-cells, respectively.}. See Figure \ref{fig:Sigma_2} (in our running examples with $g=2$,
we denote the generators of $\Gamma_{2}$ by $a,b,c,d$ instead of
$a_{1},b_{1},a_{2},b_{2}$). We identify $\Gamma_{g}$ with $\pi_{1}\left(\Sigma_{g},o\right)$,
so that in the presentation (\ref{eq:Gamma_g}), words in the generators
$a_{1},\ldots,b_{g}$ correspond to the homotopy class of the corresponding
closed paths based at $o$ along the $1$-skeleton of $\Sigma_{g}$.
Note that every covering space $p\colon\Upsilon\to\Sigma_{g}$ inherits
a CW-structure from $\Sigma_{g}$: the vertices are the pre-images
of $o$, and the open $1$-cells (2-cells) are the connected components
of the pre-images of the open 1-cells (2-cell, respectively) in $\Sigma_{g}$.

\begin{figure}
\begin{centering}
\includegraphics[viewport=120.1474bp 0bp 489bp 215.933bp,scale=0.7]{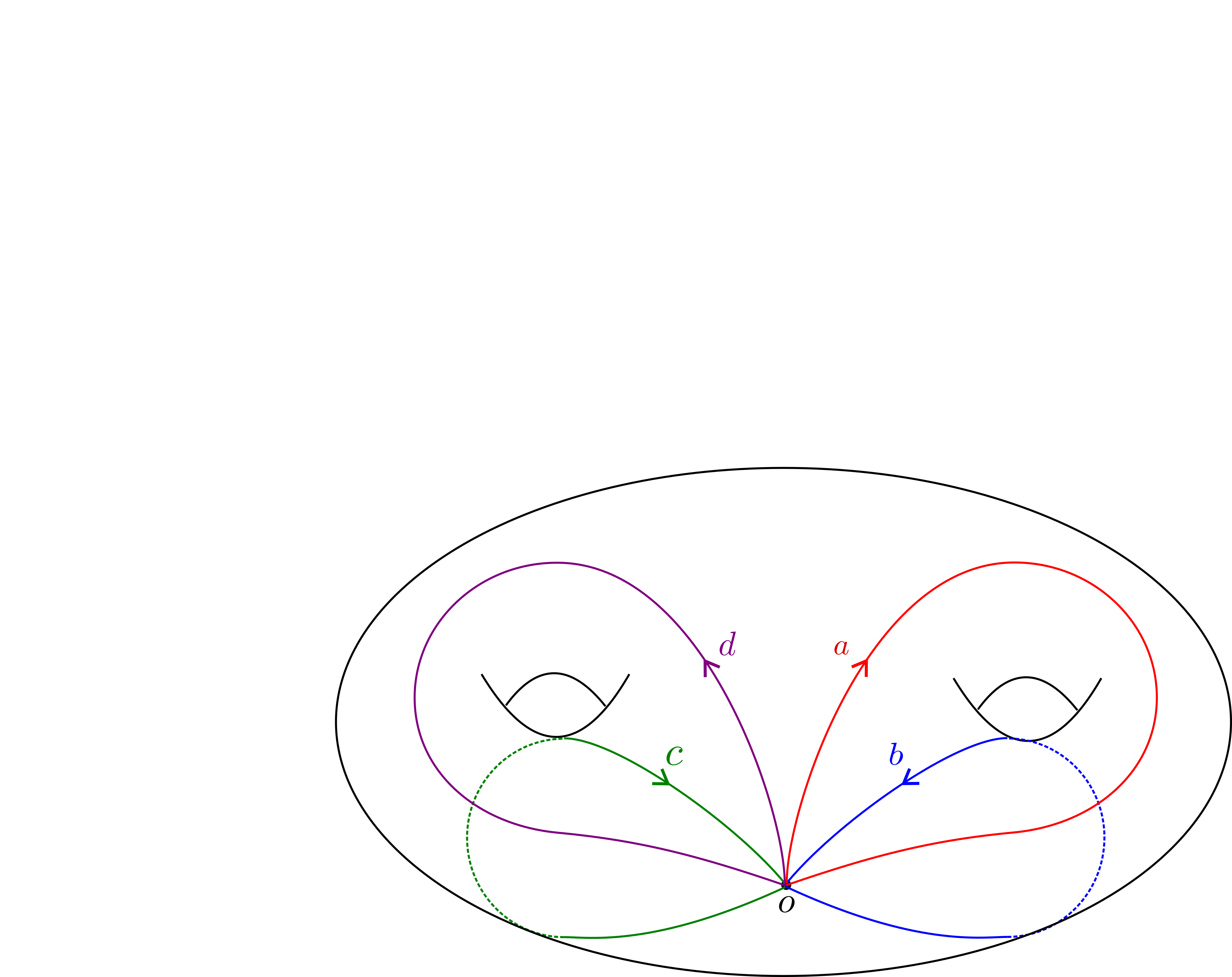}
\par\end{centering}
\caption{\label{fig:Sigma_2}The CW-structure we give to the surface $\Sigma_{2}$
with fundamental group $\left\langle a,b,c,d\,\middle|\,\left[a,b\right]\left[c,d\right]\right\rangle $:
it consists of a single vertex ($0$-cell), four edges ($1$-cells)
and one octagon (a $2$-cell).}
\end{figure}

\begin{defn}[Core surface]
\label{def:core-surface} Given a subgroup $J\le\Gamma_{g}=\pi_{1}\left(\Sigma_{g},o\right)$,
consider the covering space $p\colon\Upsilon\to\Sigma_{g}$ corresponding
to $J$. Define the \textbf{core surface of $J$}, denoted $\core\left(J\right)$,
as a sub-covering space of $\Upsilon$ as follows: $\left(i\right)$
take the union of all shortest-representative cycles in the 1-skeleton
$\Upsilon^{\left(1\right)}$ of every free-homotopy class of essential
closed curve in $\Upsilon$, and $\left(ii\right)$ add every connected
component of the complement which contains finitely many $2$-cells.

For completeness define the core surface of the trivial subgroup to
be the $0$-dimensional complex consisting of a single vertex mapped
to $o$.
\end{defn}

We define $\core\left(J\right)$ as a subcomplex of a covering space
of $\Sigma_{g}$, but we usually think of it as an at-most 2-dimensional
$CW$-complex with 1-cells that are directed and labeled by $a_{1},b_{1},\ldots,a_{g},b_{g}$.
These directions and labels on every $1$-cell completely determine
the restricted covering map. Three core surfaces are illustrated in
Figure \ref{fig:core surfaces - examples} and two others in Figure
\ref{fig:SBR may not terminate}.

\begin{figure}
\begin{centering}
\includegraphics[viewport=0bp 0bp 496bp 155.9673bp]{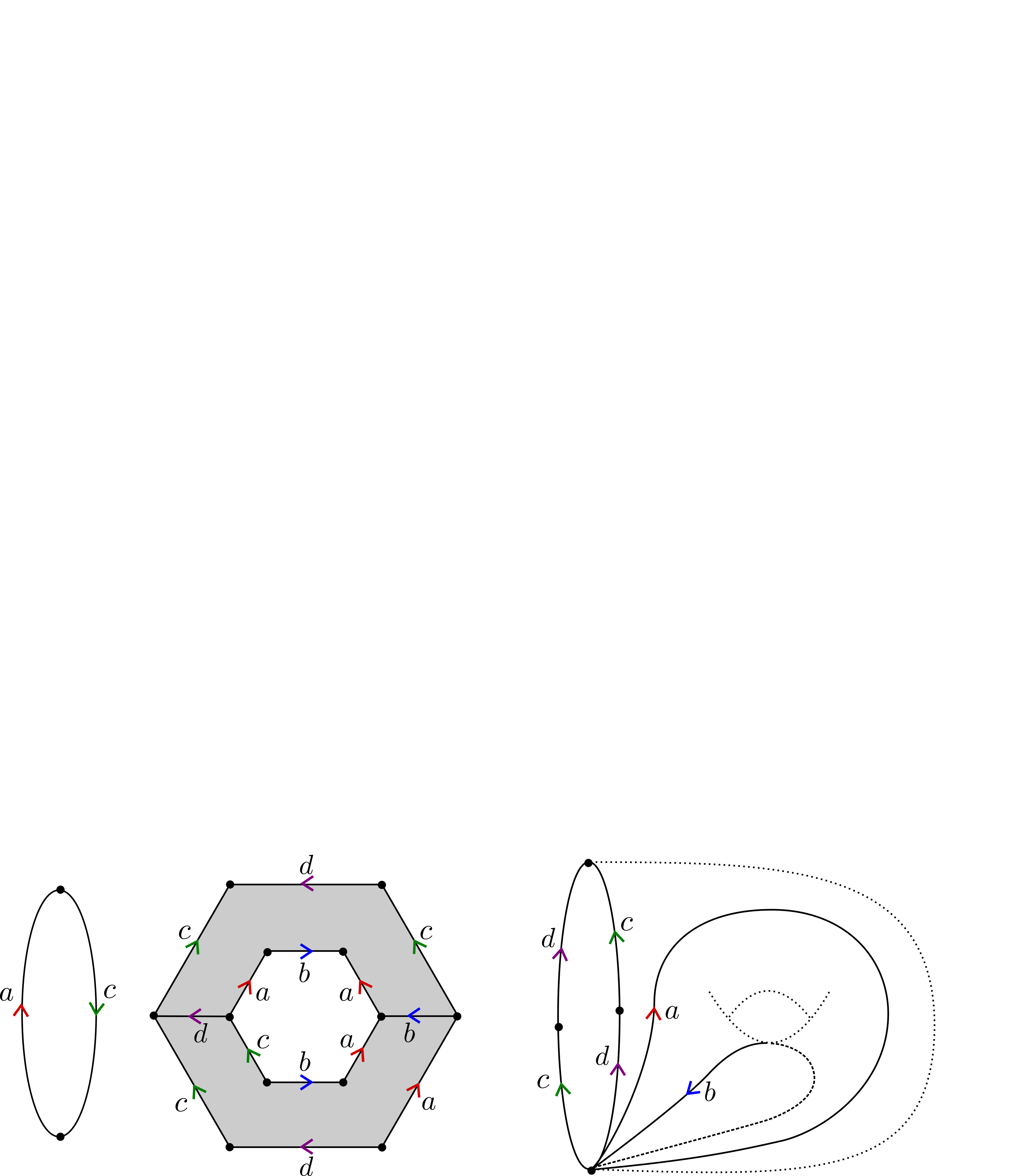}
\par\end{centering}
\caption{\label{fig:core surfaces - examples}Consider $\Gamma_{2}=\left\langle a,b,c,d\,\middle|\,\left[a,b\right]\left[c,d\right]\right\rangle $.
On the left is the core surface $\protect\core\left(\left\langle ac\right\rangle \right)$.
It consists of two vertices, two edges and no $2$-cells. The middle
object is $\protect\core\left(\left\langle aba^{-2}b^{-1}c\right\rangle \right)$,
consisting of $12$ vertices, $14$ edges and two octagons. Topologically
it is an annulus. On the right is the core surface $\protect\core\left(\left\langle a,b\right\rangle \right)$.
It consists of four vertices, six edges and one octagon, and topologically
it is a genus-1 torus with one boundary component.}
\end{figure}

After having fixed the representation (\ref{eq:Gamma_g}) for $\Gamma_{g}$,
the core surfaces are unique for every conjugacy class of subgroups.
Below, we give an intrinsic description of a core-surface which allows
one to identify a core surface without knowledge of the full covering
space it originates from (Proposition \ref{prop:intrinsic def of core surface}),
we show how to construct the core surface of $J$ from a set of generators
using a ``folding'' process (Theorem \ref{thm:folding}), and prove
a one-to-one correspondence between core surfaces and conjugacy classes
of subgroups of $\Gamma_{g}$ (Section \ref{subsec:Properties of core surfaces}).
We also show some basic properties of core surface. For instance,
we prove (Proposition \ref{prop:properties of core surfaces}) that
$\core\left(J\right)$ is connected and that it is a retract of the
covering space $\Upsilon$, and show (Proposition \ref{prop:core surfaces of fg groups are compact})
that it is compact whenever $J$ is a f.g.~subgroup. We also prove
(Lemma \ref{lem:morphisms of core surfaces}) that whenever $H\le J$
there is a natural morphism $\core\left(H\right)\to\core\left(J\right)$.

\subsubsection*{Random coverings of surfaces}

We were led to the concept of core surfaces by our work on random
homomorphisms from $\Gamma_{g}$ to the symmetric group $S_{N}$ \cite{magee2020asymptotic},
as part of a project on spectral gaps in random covering spaces of
a fixed hyperbolic surface (see \cite{magee2022random}). Within this
work we use core surfaces and the other concepts in the current paper
to prove a theorem which is parallel to some extent to Takahasi's
theorem for free groups. 
\begin{thm}
\cite{magee2020asymptotic}\label{thm:like-takahasi} Let $J\le\Gamma_{g}$
be finitely generated. Then there is a finite list of subgroups $H_{1},\ldots,H_{r}\le\Gamma_{g}$
with a fixed pointed sub-surface $\left(Y_{i},y_{i}\right)$ which
is a deformation retract of the covering space of $\Sigma_{g}$ corresponding
to $H_{i}$, so that $\left(i\right)$ $J\le H_{i}$ for all $i$,
and $\left(ii\right)$ for every subgroup $J\le L\le\Gamma_{g}$,
there is exactly one $1\le i\le r$ such that $\left(Y_{i},y_{i}\right)$
is embedded in the pointed covering of $\Sigma_{g}$ corresponding
to $L$. Moreover, this embedding is $\pi_{1}$-injective. 
\end{thm}

This theorem is essentially Theorem 2.14 and Proposition 2.15 in \cite{magee2020asymptotic},
but see also some clarifications in \cite[Section 3.3]{Zimhoni}.
The main goal of \cite{magee2020asymptotic} is to study the average
number of elements in $\left\{ 1,\ldots,N\right\} $ which are fixed
by all permutations in $\theta\left(J\right)$ when $\theta$ is a
random homomorphism $\theta\colon\Gamma_{g}\to S_{N}$. Theorem \ref{thm:like-takahasi}
is used in \cite{magee2020asymptotic} to find the asymptotics of
this number as $N\to\infty$. The main ingredient here, which hints
to how Theorem \ref{thm:like-takahasi} is used, is to show that if
$J\le\Gamma_{g}$ is f.g., then the expected number of embeddings
of $\core\left(J\right)$ into a random $N$-sheeted covering space
of $\Sigma_{g}$ is $N^{\chi\left(J\right)}\left(1+O\left(N^{-1}\right)\right)$.
\begin{rem}
Some of the content of this paper first appeared as part of the first
version of \cite{magee2020asymptotic}. Later on we decided to split
that paper into two in order to make it shorter and as we believe
the current paper is interesting for its own sake and to a potentially
different audience. We also significantly expanded the content of
the current paper.
\end{rem}

\begin{rem}
There have been different successful attempts at generalizing the
concepts of Stallings core graphs and of Stallings folds. A non-exhaustive
list includes \cite{arzhantseva1996generality,arzhantseva1998generic,kharlampovich2017stallings,brown2017geometric,beeker2018stallings,dani2021subgroups,ben2022folding}
(see \cite[Section 1]{delgado2022list} for a more exhaustive one).
Some of these works include surface groups as special cases: \cite{kharlampovich2017stallings}
introduces Stallings graphs for a large family of groups which includes
surface groups using certain types of languages of representatives
of the elements of the groups; and \cite{beeker2018stallings,dani2021subgroups,ben2022folding}
introduce Stallings-like techniques for fundamental groups of CAT$\left(0\right)$
cube complexes, a family that includes, again, hyperbolic surface
groups. Moreover, the folding process we suggest in Section \ref{subsec:Foldings-and-construction}
has the same rough structure as the one introduced in \cite{dani2021subgroups,ben2022folding}.
However, because our generalization of Stallings core graph is much
more specific and hands-on, we believe that our approach is more natural
for surface groups. In particular, as mentioned above, our core surfaces
appear quite naturally in the study of random coverings of surfaces.
\end{rem}

\subsubsection*{Notation}

We denote by $Y^{\left(1\right)}$ the $1$-skeleton of a CW-complex
$Y$. We let $g^{G}$ denote the conjugacy class of the element $g$
in a group $G$. Below, $\mathbb{Y}$ will sometimes denote the thick
version of a tiled surface, $\partial Y$ its boundary and $\left|\partial Y\right|$
the length of $\partial Y$ (Definition \ref{def:thick-version}).
The notation $Y_{+}$ refers to a tiled surface with hanging half-edges
(see page \pageref{Def of Y+}). The universal cover of the surface
$\Sigma_{g}$ is denoted $\tsg$, and for a path ${\cal P}$ or cycle
${\cal C}$ in $Y^{\left(1\right)}$, we let ${\cal P}^{*}$ and ${\cal C}^{*}$,
respectively, denote their inverses.

\subsubsection*{Paper outline}

In Section \ref{sec:core-graphs} we recall the notion of Stallings
core graphs in a more precise manner than above. In the following
two sections, before studying core surfaces per se, we introduce three
more general types of combinatorial surfaces which are sub-complexes
of covering spaces of $\Sigma_{g}$. The most general concept is that
of \emph{tiled surfaces }which is described in Section \ref{sec:Tiled-surfaces}.
This is followed by the more restricted classes of boundary reduced
and strongly boundary reduced tiled surfaces, which are defined and
analyzed in Section \ref{sec:BR and SBR}. These three classes are
natural in our context and are also important in our work on random
covering surfaces. Finally, Section \ref{sec:core surfaces} returns
to core surfaces and proves many of their properties. 

\subsection*{Acknowledgments }

This project has received funding from the European Research Council
(ERC) under the European Union’s Horizon 2020 research and innovation
programme (grant agreement No 850956 and grant agreement No 949143).
It was also supported by the Israel Science Foundation: ISF grant
1071/16.

\section{Stallings Core graphs\label{sec:core-graphs}}

Let $X=\left\{ x_{1},\ldots,x_{r}\right\} $ be a basis of the free
group $\F_{r}$, and consider the bouquet $B_{r}$ of $r$ circles
with distinct labels from $X$ and arbitrary orientations and with
a wedge point $o$. Then $\pi_{1}\left(B_{r},o\right)$ is naturally
identified with $\F_{r}$. An $X$-graph $G$ is then a graph equipped
with a graph morphism $G\to B_{r}$ which is an immersion, namely,
it is locally injective. Equivalently, $G$ is a directed graph with
edges labeled by the elements of $X$, such that no vertex admits
two outgoing edges with the same label nor two incoming edges with
the same label. An $X$-graph $G$ is an ($X$-labeled) core graph
if it is connected and if every vertex belongs to some cyclically
reduced cycle. If $G$ is finite, the latter is equivalent to $G$
being connected and having no leaves. (One usually also considers
the isolated vertex graph to be a core graph.) Multiple edges between
two vertices and loops at vertices are allowed.

There is a natural one-to-one correspondence between finite $X$-labeled
core graphs and conjugacy classes of f.g.~subgroups\footnote{Sometimes core graphs are defined with a basepoint, which is allowed
to be a leaf. Then, the correspondence is between core graphs and
subgroups, not conjugacy classes of subgroups. We present here the
non-based version because this is the version we think is more elegant
in the definition of core surfaces.} of $\F_{r}$. Indeed, given a core graph $G$ as above, pick an arbitrary
vertex $v$ and consider the ``labeled fundamental group'' $\pi_{1}^{\mathrm{lab}}\left(G,v\right)$:
closed paths in a graph with oriented and $X$-labeled edges correspond
to words in the elements of $X$. In other words, if $p\colon G\to B_{r}$
is the immersion, then $\pi_{1}^{\mathrm{lab}}\left(G,v\right)$ is
the subgroup $p_{*}\left(\pi_{1}\left(G,v\right)\right)$ of $\pi_{1}\left(B_{r},o\right)=\F_{r}$.
The conjugacy class of $\pi_{1}\left(G,v\right)$ is independent of
the choice of $v$ and is the conjugacy class corresponding to $G$. 

Conversely, if $H\le\F_{r}$ is a non-trivial f.g.~subgroup, the
conjugacy class $H^{\F_{r}}$ corresponds to a finite core graph,
denoted $G_{X}\left(H\right)$, which can be obtained in several equivalent
manners. For example, let $\Upsilon$ be the topological covering
space of $B_{r}$ corresponding to $H^{\F_{r}}$, which is equal in
this case to the Schreier graph depicting the action of $\F_{r}$
on the right cosets of $H$ with respect to the generators $X$. Then
$G_{X}\left(H\right)$ is obtained from $\Upsilon$ by 'pruning all
hanging trees', or, equivalently, as the union of all non-backtracking
cycles in $\Upsilon$.\footnote{The only exception is when $H$ is the trivial group, in which case
$\Upsilon$ is a tree and we define $\Gamma_{B}$ to consist of a
single vertex and no edges.} One can also construct $G_{X}\left(H\right)$ from any finite generating
set of $H$ using ``Stallings foldings''. Finally, a core graph
morphism is a graph morphism which commutes with the immersions to
$B_{r}$. Given two subgroups $H,J\le\F_{r}$, there is a core graph
morphism $G_{X}\left(H\right)\to G_{X}\left(J\right)$ if and only
if some conjugate of $H$ is a subgroup of $J$. See \cite{stallings1983topology,kapovich2002stallings,puder2014primitive,PP15,delgado2022list}
for more details about foldings and about core graphs in general.
See also \cite{hanany2020word} for the category of not-necessarily-connected
core graphs. As we show in the coming sections, many of the basic
properties of core graphs hold for core surfaces as well. 

\section{Tiled surfaces\label{sec:Tiled-surfaces}}

\subsection{Basic definitions}

We first define an object called a ``tiled surface'' which is the
analog of an ``$X$-graph'' from Section \ref{sec:core-graphs}.
Recall that $\Sigma_{g}$ is equipped with a CW-structure with one
vertex, $2g$ edges and a single $4g$-gon, and that this CW-structure
can be pulled back to every covering space of $\Sigma_{g}$. A sub-complex
of a CW-complex is a subspace consisting of cells such that if some
cell is in the subcomplex, than so are the cells of smaller dimension
at its boundary. 
\begin{defn}[Tiled surface]
\label{def:tiled-surface}A \emph{($g$-) tiled surface} $Y$ is
a sub-complex of a (not-necessarily-connected) covering space of $\Sigma_{g}$.
In particular, a tiled surface is equipped with the restricted covering
map $p\colon Y\to\Sigma_{g}$ which is an immersion. 
\end{defn}

Alternatively, instead of considering a tiled surface $Y$ to be a
complex equipped with a restricted covering map, one may consider
$Y$ to be a complex as above with directed and labeled edges: the
directions and labels ($a_{1},b_{1},\ldots,a_{g},b_{g}$) are pulled
back from $\Sigma_{g}$ via $p$. These labels uniquely determine
$p$ as a combinatorial map between complexes. Figures \ref{fig:Sigma_2},
\ref{fig:core surfaces - examples} and \ref{fig:SBR may not terminate}
feature examples of tiled surfaces.

Note that a tiled surface is neither necessarily compact nor necessarily
connected. Also note that a tiled surface is not always a surface:
it may contain, for example, vertices or edges with no $2$-cells
incident to them. However, as $Y$ is a sub-complex of a covering
space of $\Sigma_{g}$, namely, of a surface, any neighborhood of
$Y$ inside the covering is a surface, and it is sometimes beneficial
to think of $Y$ as such. 
\begin{defn}[Thick version of a tiled surface]
\label{def:thick-version} Given a tiled surface $Y$ which is a
subcomplex of the covering space $\Upsilon$ of $\Sigma_{g}$, consider
a small, closed, regular neighborhood of $Y$ in $\Upsilon$. The
resulting closed surface, possibly with boundary, is referred to as
the \emph{thick version of $Y$}, and is occasionally denoted $\mathbb{Y}$.
We let $\partial Y$ denote the boundary of the thick version of $Y$
and $\left|\partial Y\right|$ denote the number of edges along $\partial Y$
(so if an edge of $Y$ does not border any $4g$-gon, it is counted
twice). 
\end{defn}

In particular, $\left|\partial Y\right|=2e-4gf$ where $e$ is the
number of edges and $f$ the number of $4g$-gons in $Y$. We stress
that we do not think of $Y$ as a sub-complex, but rather as a complex
for its own sake, which happens to have the capacity to be realized
as a subcomplex of a covering space of $\Sigma_{g}$. Namely, the
underlying proper covering space in Definition \ref{def:tiled-surface}
is not part of the data possessed by $Y$. Indeed, one can also give
a direct combinatorial definition of a tiled surface. Moreover, the
thick version of a tiled surface is, too, independent of the covering
it is a subcomplex of. This is shown by the following claims.
\begin{prop}[An intrinsic definition of a tiled surface]
\label{prop:comb-def-of-tiled-surface} The following definition
is equivalent to Definition \ref{def:tiled-surface}: Fix an integer
$g\ge2$. A ($g$-) tiled surface $Y$ is an at-most two-dimensional
$\mathrm{CW}$-complex with an assignment of both a direction and
a label in $\{a_{1},b_{1},\ldots,a_{g},b_{g}\}$ to each edge, such
that:
\begin{itemize}
\item \textbf{P1: }Every vertex of $Y$ has at most one incoming $\lab$-labeled
edge and at most one outgoing $\lab$-labeled edge, for each $\lab\in\{a_{1},b_{1},\ldots,a_{g},b_{g}\}$\@. 
\item \textbf{P2: }Every path in the $1$-skeleton $Y^{\left(1\right)}$
of $Y$ reading a word in $\left\{ a_{1}^{\pm1},b_{1}^{\pm1},\ldots,a_{g}^{\pm1},b_{g}^{\pm1}\right\} $
which equals the identity in $\Gamma_{g}$ must be closed\footnote{Here, we read $a_{i}$ if we traverse an $a_{i}$-labeled edge in
its correct direction, and $a_{i}^{-1}$ if we traverse an $a_{i}$-labeled
edge against its direction, and so on.}. 
\item \textbf{P3: }Every 2-cell in $Y$ is a $4g$-gon glued along a closed
path reading the relation $\left[a_{1},b_{1}\right]\ldots\left[a_{g},b_{g}\right]$
of $\Gamma_{g}$, and every such closed path is the boundary of at
most one such $4g$-gon.
\end{itemize}
\end{prop}

Note that property \textbf{P1 }is equivalent to that $Y^{\left(1\right)}$
is an $\left\{ a_{1},\ldots,b_{g}\right\} $-graph, in the sense of
Section \ref{sec:core-graphs}.
\begin{proof}
It is straightforward that a tiled surface from Definition \ref{def:tiled-surface}
satisfies the assumptions in the statement of the proposition. Conversely,
let $Y$ be a $CW$-complex as in the statement of the proposition.
It is enough to show that every connected component of $Y$ is a tiled
surface à la Definition \ref{def:tiled-surface}, so assume without
loss of generality that $Y$ is connected and non-empty. Define a
map $p\colon Y\to\Sigma_{g}$ by sending every vertex of $Y$ to $o$,
every edge of $Y$ to the corresponding edge of $\Sigma_{g}$ (with
the same label and corresponding direction), and every $4g$-gon of
$Y$ to the single $4g$-gon of $\Sigma_{g}$. This map is a well-defined
map of $CW$-complexes.

Now pick an arbitrary vertex $v$ in $Y$. Denote $J\eqdf p_{*}\left(\pi_{1}\left(Y,v\right)\right)\le\pi_{1}\left(\Sigma_{g},o\right)=\Gamma_{g}$.
Let $q\colon\Upsilon\to\Sigma_{g}$ be the connected covering space
corresponding to $J$, and let $u$ be a vertex of $\Upsilon$ so
that $q_{*}\left(\pi_{1}\left(\Upsilon,u\right)\right)=J$. By standard
facts from the theory of covering spaces \cite[Propositions 1.33 and 1.34]{hatcher2005algebraic},
as $p_{*}\left(\pi_{1}\left(Y,v\right)\right)\subseteq q_{*}\left(\pi_{1}\left(\Upsilon,u\right)\right)$,
there is a unique lift $r\colon\left(Y,v\right)\to\left(\Upsilon,u\right)$
of $p$ so that $q\circ r=p$.
\begin{equation}
\xymatrix{ & \left(\Upsilon,u\right)\ar@{->>}[d]^{q}\\
\left(Y,v\right)\ar[r]_{p}\ar@{-->}[ur]^{\exists!r} & \left(\Sigma_{g},o\right)
}
\label{eq:lift}
\end{equation}
Clearly, $r$ respects the $CW$-structure of $Y$ and $\Upsilon$
as well as the labels and directions of edges. It remains to show
that $r$ is injective, for then $Y$ is indeed a subcomplex of $\Upsilon$
and therefore a tiled surface. 

So assume towards contradiction that $r$ is not injective. By properties
\textbf{P1} and \textbf{P3}, there must be two distinct vertices of
$Y$ with the same $r$-image. Without loss of generality we may assume
that one of them is $v$ (otherwise replace $v$ with this vertex
and replace $u$ accordingly: this does not change $q$ nor $r$).
Call the other vertex $v'$. But then, any path in $Y^{\left(1\right)}$
from $v$ to $v'$ is mapped to a closed path at $u$ and therefore
reads a word $w_{1}$ in $\left\{ a_{1}^{\pm1},\ldots,b_{g}^{\pm1}\right\} $
representing some element $j\in J$. As $p_{*}\left(\pi_{1}\left(Y,v\right)\right)=J$,
there is a closed path at $v$ reading $w_{2}$ with $w_{2}$ also
representing the same $j$. Then there exists a path from $v$ to
$v'$ reading $w_{2}^{-1}w_{1}$ which is the identity in $\Gamma_{g}$,
a contradiction to Property \textbf{P2}.
\end{proof}
The combinatorial definition of a tiled surface given in Proposition
\ref{prop:comb-def-of-tiled-surface} is not very satisfactory as
it involves a ``global'' condition -- Property \textbf{P2} --
which is not easy to check, at least not complexity-wise\footnote{Although if $Y$ is a finite complex, it is possible to check condition
\textbf{P2} by the classical solution of Dehn to the word problem
in $\Gamma_{g}$ \cite{Dehn}. See also Theorem \ref{thm:birman-series}.}. In Section \ref{sec:BR and SBR} we present a restricted version
of tiled surfaces -- boundary reduced tiled surfaces -- where \textbf{P2}
can be replaced with a local property. Before that, we show that the
thick version of a tiled surface can also be extracted from the combinatorial,
covering-space-free, definition.
\begin{prop}[An intrinsic definition of the thick version of a tiled surface]
\label{prop:comb-def-of-thick-version} The thick version from Definition
\ref{def:thick-version} of a tiled surface $Y$ can be extracted
from the combinatorial data of $Y$, namely, from the CW-structure
with directions and labels of edges as in Proposition \ref{prop:comb-def-of-tiled-surface}.
This can be done as follows:
\begin{itemize}
\item Around every vertex, give the half-edges incident to the vertex the
cyclic order induced from the cyclic order of half-edges around the
vertex $o$ in $\Sigma_{g}$. Namely, this is the cyclic order which
is (say, clockwise) a cyclic subsequence of $a_{1}$-outgoing, $b_{1}$-incoming,
$a_{1}$-incoming, $b_{1}$-outgoing, $a_{2}$-outgoing, $b_{2}$-incoming,
$a_{2}$-incoming, $b_{2}$-outgoing and so on.
\item The cyclic ordering of half-edges at each vertex makes $Y^{\left(1\right)}$
into a \emph{ribbon graph }that yields an orientable surface with
boundary\emph{}\footnote{A ribbon graph is also known as a fat-graph. The surface is obtained
by thickening every vertex to a disc and every edge to a strip.}\emph{. }
\item Every $4g$-gon of $Y$ corresponds to some boundary component of
the ribbon graph (which reads the relation $\left[a_{1},b_{1}\right]\cdots\left[a_{g},b_{g}\right]$),
and the thick version of $Y$ is then obtained from the ribbon graph
by attaching a $4g$-gon along every boundary component corresponding
to a $2$-cell of $Y$.
\end{itemize}
\end{prop}

\begin{proof}
If a tiled surface $Y$ is defined via Definition \ref{def:tiled-surface}
as a sub-complex of a covering space $\Upsilon$ of $\Sigma_{g}$,
then every vertex inherits a cyclic ordering of the incident half-edges
from the ordering in $\Upsilon$, which is always the ordering specified
in the statement of the proposition. The thick version of $Y^{\left(1\right)}$
is therefore precisely the ribbon graph described in the statement
of the proposition, and the $4g$-gons are attached as described there
too. Finally, because the description in the statement of the proposition
is well defined and therefore unique, it must indeed recover the thick
version from Definition \ref{def:thick-version}.
\end{proof}

\begin{example}[Universal cover of $\Sigma_{g}$]
\label{exa:universal-cover} Let $\widetilde{\Sigma_{g}}$ denote
the universal cover of $\Sigma_{g}$ endowed with the CW-structure
pulled back from $\Sigma_{g}$ via the covering map. As a topological
space, $\widetilde{\Sigma_{g}}$ is an open disc. There is a natural
action of $\Gamma_{g}$ on $\widetilde{\Sigma_{g}}$ by isomorphisms
of tiled surfaces such that $\Gamma_{g}\backslash\widetilde{\Sigma_{g}}=\Sigma_{g}$.
We fix, once and for all, an arbitrary vertex $u$ in $\widetilde{\Sigma_{g}}$,
to obtain a pointed tiled surface $(\widetilde{\Sigma_{g}},u)$. Note
that $\left(\widetilde{\Sigma_{g}},u\right)$ is the Cayley complex
of $\Gamma_{g}$ and its $1$-skeleton is the Cayley graph of $\Gamma_{g}$
with respect to the generators $a_{1},\ldots,b_{g}$. For every $J\le\Gamma_{g}$,
the covering space of $\Sigma_{g}$ corresponding to $J$ can be also
defined as $J\backslash\widetilde{\Sigma_{g}}$.
\end{example}

\subsubsection*{Morphisms of tiled surfaces}

If $Y_{1}$ and $Y_{2}$ are tiled surfaces, a \emph{morphism} from
$Y_{1}$ to $Y_{2}$ is a map of $CW$-complexes which maps $i$-cells
to $i$-cells for $i=0,1,2$ and respects the directions and labels
of edges. Equivalently, this is a morphism of CW-complexes which commutes
with the restricted covering maps $p_{j}\colon Y_{j}\to\Sigma_{g}$
($j=1,2$). In particular, the restricted covering map from a tiled
surface to $\Sigma_{g}$ is itself a morphism of tiled surfaces. It
is an easy observation that every morphism of tiled surfaces is an
immersion (locally injective). 

\subsection{Boundary cycles, hanging half-edges, blocks, and chains\label{subsec:Boundary-cycles,-hanging}}

In the current Section \ref{subsec:Boundary-cycles,-hanging} we define
some notions related to tiled surfaces which will play an important
role in the coming sections, where we define (strongly) boundary reduced
tiled surfaces and analyze the properties of core surfaces.

Given a tiled surface $Y$, a \textbf{\emph{path}}\textbf{ }in $Y$
is a sequence $\PP=\left(\vec{e_{1}},\ldots,\vec{e}_{k}\right)$ of
directed edges $\vec{e_{1}},\ldots,\vec{e}_{k}$ in $Y^{(1)}$, where
for each $1\leq i\leq k-1$ the terminal vertex of $\vec{e}_{i}$
is the initial vertex of $\vec{e}_{i+1}$. The direction of the edges
in the cycle is not necessarily the same as the direction dictated
in the definition of the tiled surface. A \textbf{\emph{cycle}}\emph{
}in $Y$ is a path where, in addition, the terminal vertex of $\vec{e}_{k}$
is the initial vertex of $\vec{e}_{1}$. We write $\PP^{*}$ or $\mathcal{C}^{*}$
for the oppositely oriented path or cycle, respectively, $(\vec{e}_{k}^{*},\ldots,\vec{e}_{1}^{*})$
where $\vec{e}_{i}^{*}$ is $\vec{e}_{i}$ with the opposite direction.

Every cycle $\CC$ yields a cyclic word $w(\mathcal{C})$ in $\left\{ a_{1}^{\pm1},\ldots,b_{g}^{\pm1}\right\} $
by reading the label $\lab_{i}\in\left\{ a_{1},\ldots,b_{g}\right\} $
of the edge $\vec{e}_{i}$, $1\leq i\leq k$ in order and writing
(from left to right) $\lab_{i}$ if the direction of $\vec{e_{i}}$
is the same as the given direction in the tiled surface $Y$, and
$\lab_{i}^{-1}$ otherwise. Every word $w$ in $\left\{ a_{1}^{\pm1},\ldots,b_{g}^{\pm1}\right\} $
corresponds to an element $\gamma_{w}\in\Gamma_{g}$ via the presentation
(\ref{eq:Gamma_g}). We write $\gamma^{\Gamma_{g}}$ for the conjugacy
class of $\gamma$ in $\Gamma_{g}$. We say that $w$ represents the
conjugacy class $\gamma_{w}^{~\Gamma_{g}}$. 

We denote by $\ell(w)$ the length of a word $w$, and if $\gamma\in\Gamma_{g}$,
we write $\ell(\gamma^{\Gamma_{g}})$ for the minimal length of a
word $w$ for which $\gamma_{w}^{~\Gamma_{g}}=\gamma^{\Gamma_{g}}$.
For $\gamma\in\Gamma_{g}$, we say that $\gamma$ is \textbf{\emph{cyclically
shortest}}\emph{ }if $\gamma=\gamma_{w}$ for some word $w$ with
$\ell(w)=\ell(\gamma^{\Gamma_{g}})$. We say that a word $w$ is \textbf{\emph{cyclically
shortest}}\emph{ }if $\ell(w)=\ell(\gamma_{w}^{~\Gamma_{g}})$.\\

In the rest of the paper, we wish to use some notions of Birman and
Series from \cite{BirmanSeries}. However, we make one small adjustment:
what Birman and Series call a cycle, we will call a block\footnote{This is so that we can reserve the term cycle to be used in the usual
way as we have above.}.

We wish to augment the tiled surface $Y$ by adding some new half-edges.
Here, formally, a half-edge is a copy of the interval $[0,\frac{1}{2})$.
Every edge of $Y$ gives rise to two half-edges which cover the entire
edge except for the midpoint. We may also add a new half-edge to a
vertex, in which case the point $0$ will be identified with the vertex.
For every vertex $v$ of $Y$, we add at most $8$ half-edges to $v$
to form a new surface \label{Def of Y+}$Y_{+}$. The half-edges are
added in such a way that the morphism $p\colon Y\to\Sigma_{g}$ extends
to a map $p_{+}:Y_{+}\to\Sigma_{g}$ that gives a local homeomorphism
of $1$-skeleta at each vertex of $Y_{+}$. We call a half-edge of
$Y_{+}$ a \textbf{\emph{hanging half-edge}} if it is added in this
way: the remaining half-edges in $Y_{+}$ are contained in proper
edges of $Y$.

As a result, there are now exactly $4g$ half-edges incident to every
vertex in $Y_{+}$. The hanging half-edges of $Y_{+}$ inherit both
a label in $\{a_{1},\ldots,b_{g}\}$ and a direction from the corresponding
half-edges in $\Sigma_{g}$. Moreover, at each vertex of $Y_{+}$,
the incident half-edges have a cyclic ordering given by the (clockwise)
cyclic ordering of the half-edges in $\Sigma_{g}$. We fix these labels,
directions, and the cyclic ordering of half-edges at each vertex as
part of the data of $Y_{+}$.

Given two directed edges $\vec{e}_{1}$ and $\vec{e}_{2}$ of $Y$,
with the terminal vertex $v$ of $\vec{e}_{1}$ equal to the source
vertex of $\vec{e}_{2}$, we refer to the $m$ half-edges of $Y_{+}$
incident to $v$ between \emph{$\vec{e}_{1}$} and \emph{$\vec{e}_{2}$}
in the given cyclic order of the $4g$ half-edges at $v$ as \emph{the
}\textbf{\emph{half-edges between}}\emph{ $\vec{e}_{1}$ and $\vec{e}_{2}$}.
Here $0\le m\le4g-1$.

\begin{figure}
\begin{centering}
\includegraphics{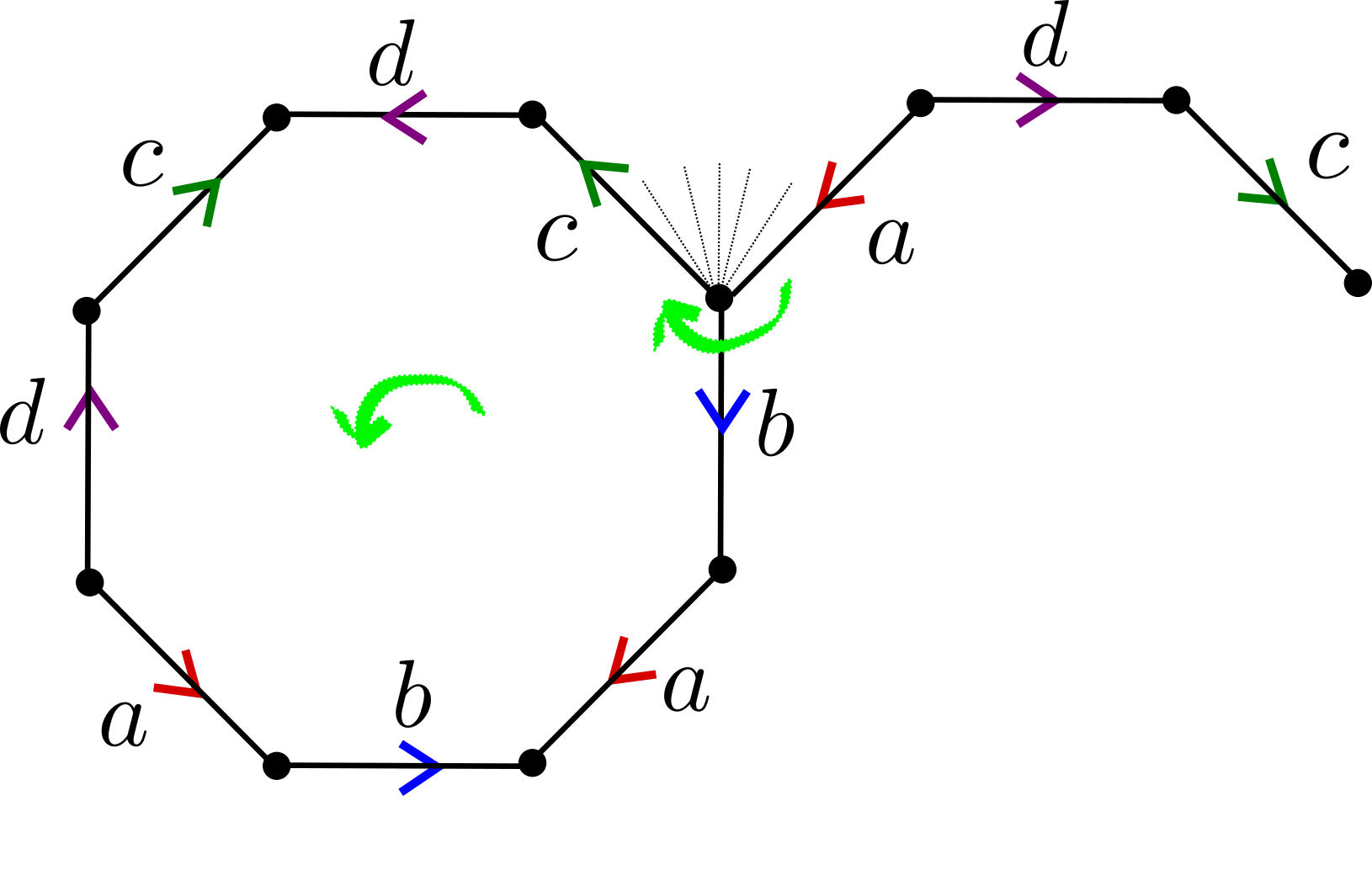}
\par\end{centering}
\caption{\label{fig:blocks and orientation}Fix $g=2$. The figure shows a
piece of a tiled surface containing one octagon and three more edges
outside it. Blocks always follow the boundary of an octagon along
the preset orientation (counter-clockwise in our figures), while half-edges
around a vertex are ordered clockwise. For instance, any path reading
$c^{-1}d^{-1}a$ is a block. In the figure there is a block $c^{-1}d^{-1}a$
on the right and a block $cdc^{-1}d^{-1}a$ on the left (the latter
one can be extended on both ends). These two blocks are consecutive:
there is exactly one half-edge (outgoing $b$) between the last edge
(incoming $a$) of the first block and the first edge (outgoing $c$)
of the second block. Therefore, their concatenation forms a chain
$c^{-1}d^{-1}acdc^{-1}d^{-1}a$.}
\end{figure}

A path in a tiled surface $Y$ is a \textbf{\emph{block}}\emph{ }if
it is non-empty and each pair of successive edges have no half-edges
between them. A block runs along the boundary of a single $4g$-gon
in $Y$ or a single ``phantom-$4g$-gon'' which exists in the covering
space of $\Sigma_{g}$ of which $Y$ is a sub-complex. In other words,
a block is a path that reads a subword of the cyclic word $\left[a_{1},b_{1}\right]\ldots\left[a_{g},b_{g}\right]$.
Note that the inverse of a block of length $\ge2$ is \emph{not }a
block. 

A \textbf{\emph{half-block}} is a block of length $2g$, and a \textbf{\emph{long
block}} is a block of length at least $2g+1$, including the case
of a ``full block'' of length $4g$. If a (non-cyclic) block of
length $b$ sits along the boundary of a $4g$-gon $O$, the \textbf{\emph{complement}}\textbf{
of the block }is the inverse of the block of length $4g-b$ consisting
of the complement set of edges along $O$ (so the block and its complement
share the same starting point and the same terminal point).

We say that two blocks $(\vec{e}_{i},\ldots,\vec{e}_{j})$ and $(\vec{e}_{k},\ldots,\vec{e}_{\ell})$
are \textbf{\emph{consecutive}}\emph{ }if $(\vec{e}_{i},\ldots,\vec{e}_{j},\vec{e}_{k},\ldots,\vec{e}_{\ell})$
is a path and there is \uline{exactly} one half-edge between $\vec{e}_{j}$
and $\vec{e}_{k}$. A \textbf{\emph{chain}} is a (possibly cyclic)
sequence of consecutive blocks. This is illustrated in Figure \ref{fig:blocks and orientation}.
A \textbf{\emph{cyclic chain}}\textbf{ }is a chain whose blocks pave
an entire cycle. A \textbf{\emph{long chain}} is a chain consisting
of consecutive blocks of lengths 
\[
2g,2g-1,2g-1,\ldots,2g-1,2g.
\]
A \textbf{\emph{half-chain}}\emph{}\footnote{This notion ``half-chain'' is ours -- it does not appear in \cite{BirmanSeries}.
While this name does not capture the essence of these objects, we
chose it because half-chains are related to half-blocks in roughly
the same manner as long chains are related to long blocks. This will
be apparent in Section \ref{sec:BR and SBR}.} is a cyclic chain consisting of consecutive blocks of length $2g-1$
each.

The \textbf{complement of a long chain }is the inverse of a chain
with blocks of lengths $2g-1,2g-1,\ldots,2g-1$ which sits along the
other side of the $4g$-gons bordering the long chain. Note that the
complement of a long chain shares the same starting point and terminal
point as the long chain, and is shorter by two edges from the long
chain. See Figure \ref{fig:a complement of a long chain}.

\begin{figure}
\begin{centering}
\includegraphics{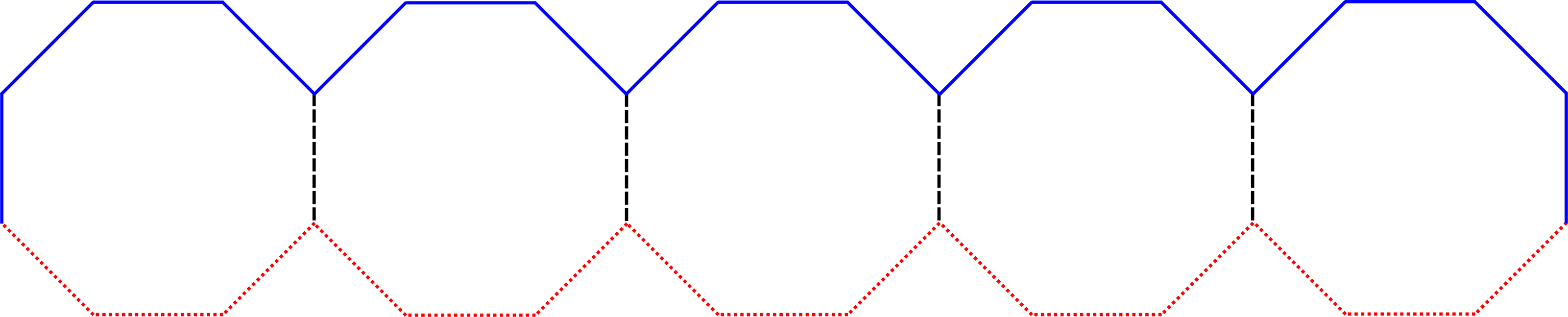}
\par\end{centering}
\caption{\label{fig:a complement of a long chain}Fix $g=2$. The figure shows
a long chain of total length $17$ (blocks of sizes $4,3,3,3,4$,
in blue) and its complement of length $15$ which is the inverse of
a chain made of blocks of lengths $3,3,3,3,3$ (in red).}
\end{figure}

The \textbf{complement of a half-chain} is defined as follows. If
the half-chain sits along the boundary of the $4g$-gons $O_{1},\ldots,O_{r}$,
its complement is the inverse of the half-chain sitting along the
other sides of these $4g$-gons: a block (of length $2g-1$) of the
half-chain along $O_{i}$ is replaced by the path of length $2g-1$
along $O_{i}$, with starting and terminal points one edge away from
the starting and terminal points, respectively, of the block. The
complement of a half-chain has the same length as the original half-chain.
The middle part of Figure \ref{fig:core surfaces - examples} illustrates
two complementing half-chains of length $6$ each (with two octagons
in between).

The following lemma shows, in particular, that there are no cyclic
chains consisting of one block of length $2g$ nor cyclic chains consisting
of blocks of lengths $2g,2g-1,2g-1,\ldots,2g-1$.
\begin{lem}
\label{lem:no cyclic long chains}In every cyclic chain, the number
of blocks of even length is even.
\end{lem}

(This excludes the special case of a cyclic chain consisting of a
single block of length $4g$, which may be excluded from the definition
of a chain.)
\begin{proof}
In the defining relation $\left[a_{1},b_{1}\right]\ldots,\left[a_{g},b_{g}\right]$,
every letter is at distance $2$ from its inverse. In two consecutive
blocks $(\vec{e}_{i},\ldots,\vec{e}_{j})$ and $(\vec{e}_{k},\ldots,\vec{e}_{\ell})$,
there are thus two possible cases. Either the letter associated with
$\vec{e}_{k}$ is identical to the letter associated with $\vec{e}_{j}$,
as in $\left(a_{1},b_{1},a_{1}^{-1}\right),\left(a_{1}^{-1},b_{1}^{-1},a_{2},b_{2}\right)$
with an incoming half $b_{1}$-edge hanging in between. Or $\vec{e}_{k}$
comes four places after $\vec{e}_{j}$ in the defining relation, as
in $\left(b_{g}^{-1},a_{1}\right),\left(a_{2},b_{2},a_{2}^{-1},b_{2}^{-1}\right)$
with an outgoing half $b_{1}$-edge hanging in between. Hence the
parity of the location in the defining relation of the first letter
in a block alternates after an even-length block. As the defining
relation has even length, this proves the lemma.
\end{proof}
For every directed edge $\vec{e}$ in a tiled surface $Y$ and every
$4g$-gon $O$ in $Y$ that meets $\vec{e}$ at its boundary, we say
that $O$ is on the left (resp.~right) of $\vec{e}$ if for a small
neighborhood $N$ of $\vec{e}$ in $O$, $N$ is on the left (resp.~right)
of $\vec{e}$ as $\vec{e}$ is traversed in its given direction, where
left/right is defined with respect to an orientation inherited from
a fixed orientation of $\Sigma_{g}$. Note that a $4g$-gon $O$ can
be both on the left and right of a directed edge if that edge appears
twice in the boundary of $O$: this is the case, for instance, with
the $a$- and $b$-edges in the right core surface in Figure \ref{fig:core surfaces - examples}.

A \textbf{\emph{boundary cycle}} of $Y$ is a cycle $(\vec{e_{1}},\ldots,\vec{e}_{k})$
in $Y$ corresponding to an oriented boundary component of the thick
version of $Y$ (see Definition \ref{def:thick-version}). We always
choose the orientation so that there are no $4g$-gons to the immediate
\textbf{left }of the boundary component as it is traversed. Therefore,
boundary components of $Y$ correspond to unique cycles. Note that
$\left|\partial Y\right|$ is equal to the sum over boundary cycles
of $Y$ of the number of edges in each such cycle. 

\section{Boundary reduced and strongly boundary reduced tiled surfaces\label{sec:BR and SBR}}

In this section we describe a restricted class of tiled surfaces called
``boundary reduced'' and its subclass of ``strongly boundary reduced''
tiled surfaces. These are tiled surfaces with ``nice'' boundary,
which turns the global property \textbf{P2 }from Proposition \ref{prop:comb-def-of-tiled-surface}
into a simpler one. As we show in the next section, this class also
contains all core surfaces, and as such the properties of its elements
are important for our main object of study. Moreover, these notions
are also important for the analysis in \cite{magee2020asymptotic}:
for a compact, (strongly) boundary reduced tiled surface we are able
to give a rather precise estimate for the expected number of times
it is embedded in a random $N$-sheeted covering of $\Sigma_{g}$
\cite[Propositions 5.26 and 5.26]{magee2020asymptotic}. 
\begin{defn}[Boundary reduced]
\label{def:br}A tiled surface $Y$ is \emph{boundary reduced} if
no boundary cycle of $Y$ contains a long block or a long chain.
\end{defn}

In particular, if $Y$ is boundary reduced, then every path that reads
$\left[a_{1},b_{1}\right]\ldots\left[a_{g},b_{g}\right]$ is not only
closed, but there is also a $4g$-gon attached to it. We also need
a stronger version of this property.
\begin{defn}[Strongly boundary reduced]
\label{def:sbr}A tiled surface $Y$ is \emph{strongly boundary reduced}
if no boundary cycle of $Y$ contains a half-block or a half-chain.
\end{defn}

Because a long block contains (at least two) half-blocks and a long
chain contains (two) half-blocks, a strongly boundary reduced tiled
surface is in particular boundary reduced. 

We now show that in the combinatorial definition of a tiled surface
(Proposition \ref{prop:comb-def-of-tiled-surface}), if the boundary
is reduced, then property \textbf{P2 }holds automatically. Note that
the combinatorial definition of the thick version of a tiled surface,
as in Proposition \ref{prop:comb-def-of-thick-version}, does not
depend on the validity of property \textbf{P2}, namely, the thick
version is well-defined even when the complex $Y$ satisfies the assumptions
of Proposition \ref{prop:comb-def-of-tiled-surface} excluding \textbf{P2}.
\begin{prop}
\label{prop:BR core graph with octagons is a tiled surface} Let $Y$
be an at-most two-dimensional $\mathrm{CW}$-complex with an assignment
of both a direction and a label in $\{a_{1},\ldots,b_{g}\}$ to each
edge, such that properties \textbf{P1} and \textbf{P3 }from Proposition
\ref{prop:comb-def-of-tiled-surface} hold. It the thick version of
$Y$ (as in Proposition \ref{prop:comb-def-of-thick-version}) is
boundary reduced, then $Y$ is a tiled surface.
\end{prop}

\begin{proof}
It is enough to prove that $Y$ satisfies \textbf{P2}. Assume toward
contradiction that there is a non-closed path in $Y^{\left(1\right)}$
reading a word which equals the identity in $\Gamma_{g}$. Let $\PP$
be such a path of minimal length. By the classical results of Dehn
\cite{Dehn}, $w\left(\PP\right)$ must then contain a subword which
is a cyclic piece of the relation $\left[a_{1},b_{1}\right]\ldots\left[a_{g},b_{g}\right]$
or its inverse of length more than half (so at least $2g+1$) (this
also follows, for example, from Greendlinger's lemma -- see \cite[Theorem V.4.5]{lyndon1977combinatorial}).
But this corresponds to a long block in $\PP$ or in $\PP^{*}$, and
as $Y$ is boundary reduced, the complement of this long block is
also in $Y$. By replacing the long block with its complement, we
obtain a strictly shorter path reading a word which equals the identity
-- a contradiction.
\end{proof}
\begin{defn}[$\br$- and $\sbr$-closure]
\label{def:BR- and SBR-closure} Let $Y$ be a tiled surface embedded
in a boundary reduced tiled surface $Z$. The \emph{boundary reduced
closure} of $Y$ in $Z$, or $\br$-closure, denoted $\br(Y\hookrightarrow Z)$,
is the intersection of all intermediate tiled surfaces $Y\hookrightarrow Y'\hookrightarrow Z$
which are boundary reduced.

Likewise, if $Z$ is strongly boundary reduced, the \emph{strongly
boundary reduced closure }of $Y$ in $Z$, or $\sbr$-closure, denoted
$\sbr\left(Y\hookrightarrow Z\right)$, is the intersection of all
intermediate strongly boundary reduced tiled surfaces $Y\hookrightarrow Y'\hookrightarrow Z$.
\end{defn}

\begin{prop}
\label{prop:(S)BR-closure is (S)BR}The $\br$-closure of $Y$ is
boundary reduced and contains $Y$. The $\sbr$-closure of $Y$ is
strongly boundary reduced and contains $Y$.
\end{prop}

\begin{proof}
By assumption, $Z$ itself is (strongly) boundary reduced, and so
the intersection is over a non-empty set of tiled surfaces. It trivially
contains $Y$. We claim that the intersection of every family of boundary
reduced tiled sub-surfaces of $Z$ is boundary reduced. Indeed, if
the intersection $X=\cap Y'$ has some long block $b$ at its boundary
$\partial X$, then $b$ is also a long block at the boundary of some
$Y'$, which is impossible. If $c$ is a long chain at $\partial X$
and $O$ some $4g$-gon of $Z\setminus X$ sitting along $c$, then
there is some $Y'$ not containing $O$ but containing $c$. The intersection
$\partial Y'\cap\partial O$ must then either include a long block
or a block which belongs to a long chain of $\partial Y'$ (see Figure
\ref{fig:an exposed piece of a long chain}), which is impossible. 

Similarly, we claim that the intersection $X=\cap Y'$ of every family
of strongly boundary reduced tiled sub-surfaces of $Z$ is strongly
boundary reduced. If $b$ is a half-block at $\partial X$, it must
also belong to $\partial Y'$ for some $Y'$, which is impossible.
If $\partial X$ contains a half-chain $c$, then every $4g$-gon
$O$ of $Z\setminus X$ sitting along $c$ does not belong to some
$Y'$. But then $O$ sits along a half block, a long chain or the
same half-chain along $\partial Y'$, which is impossible. 
\end{proof}
\begin{figure}
\begin{centering}
\includegraphics[scale=0.7]{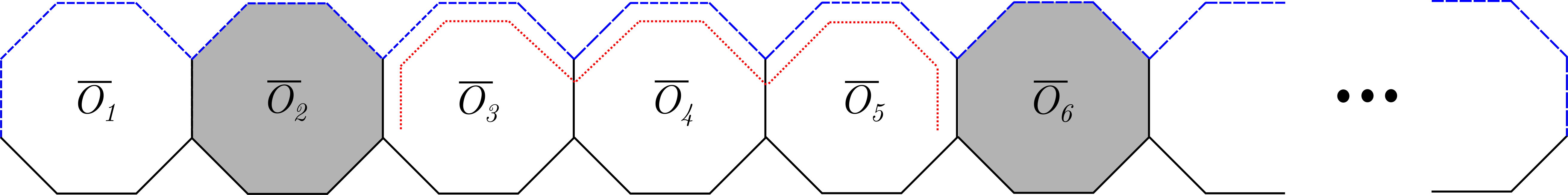}
\par\end{centering}
\caption{\label{fig:an exposed piece of a long chain}Assume $g=2$, and let
$Z$ be a boundary reduced tiled surface containing a long chain $c$
(broken blue line) bordering the octagons $O_{1},\ldots,O_{r}$ (not
necessarily distinct). If $Y$ is sub-tiled surface of $Z$ containing
$c$ but not all the octagons $O_{1},\ldots,O_{r}$, then $Y$ is
not boundary reduced. For example, if $O_{3},O_{4},O_{5}\protect\notin Y$
is a longest subsequence of octagons not in $Y$ (so $O_{2},O_{6}\in Y$),
then there is a long chain (dotted red line) at $\partial Y$ along
these three octagons.}
\end{figure}

A useful property we now prove is that the boundary-reduced closure
$\br(Y\hookrightarrow Z)$ of a compact tiled surface $Y$ is compact
too. The result in Proposition \ref{prop:BR closure of compact is compact}
is not true for $\sbr$-closure. This is illustrated in Figure \ref{fig:SBR may not terminate}.
\begin{prop}
\label{prop:BR closure of compact is compact} Let $Y$ be a compact
tiled surface embedded in a boundary reduced tiled surface $Z$. Then
$\br\left(Y\hookrightarrow Z\right)$ is compact too.
\end{prop}

\begin{proof}
Our definition of the $\br$-closure is from the top down: by taking
the intersection of boundary reduced tiled surfaces. But when $Y$
is compact, one can instead construct its $\br$-closure in $Z$ from
the bottom up by adding closed $4g$-gons to $Y$. We describe this
process and show that one needs to add only finitely many $4g$-gons,
which will prove the claim.

Indeed, let $Y'=Y$. We add (closed) $4g$-gons to $Y'$ while keeping
it a sub-tiled surface of $\br\left(Y\hookrightarrow Z\right)$. By
(the proof of) Proposition \ref{prop:(S)BR-closure is (S)BR}, as
$Y'\subseteq\br\left(Y\hookrightarrow Z\right)$, if $\partial Y'$
contains a long block, then the $4g$-gon along it must too belong
to $\br\left(Y\hookrightarrow Z\right)$, and we may add this $4g$-gon
to $Y$'. In doing so we removed at least $2g+1$ edges from $\partial Y'$,
and added to $\partial Y'$ at most $2g-1$ new edges, thus reducing
$\left|\partial Y'\right|$ by at least $2$. Likewise, if $\partial Y'$
contains a long chain, then all $4g$-gons along it must belong to
$\br\left(Y\hookrightarrow Z\right)$ and we may add them to $Y'$.
In this step $\left|\partial Y'\right|$ is reduced again by at least
two. But $\left|\partial Y'\right|$ is always non-negative hence
this process must end after finitely many steps. The resulting tiled
surface $Y'$ is boundary reduced and thus equals $\br\left(Y\hookrightarrow Z\right)$.
Because we added a finite number of $4g$-gons to a compact tiled
surface, $\br\left(Y\hookrightarrow Z\right)$ is compact too.
\end{proof}
\begin{figure}
\begin{centering}
\includegraphics[scale=0.9]{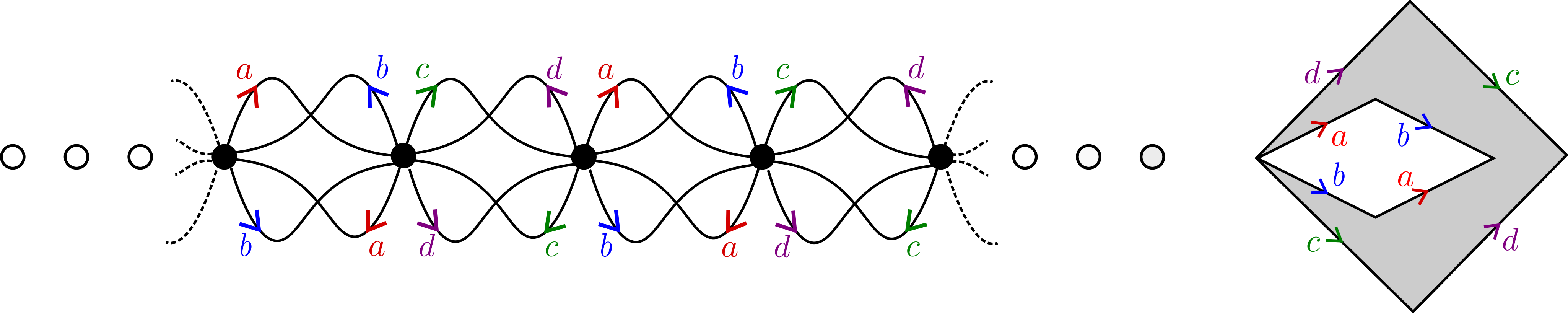}
\par\end{centering}
\caption{\label{fig:SBR may not terminate}Let $g=2$ and $\Gamma_{2}=\left\langle a,b,c,d\,\middle|\,\left[a,b\right]\left[c,d\right]\right\rangle $.
On the left hand side is a piece of an the infinite $1$-skeleton
of a tiled surface $Z$ with no boundary. The full surface extends
to both sides infinitely with the same fixed pattern, and every possible
octagon is included in it. This is the covering space of $\Sigma_{2}$,
as well as the core surface, corresponding to the normal subgroup
$\ll a^{2},ab,ab^{-1},c^{2},cd\gg\trianglelefteq\Gamma_{2}$. On the
right is a tiled surface $Y$ consisting of a single octagon with
two of its vertices identified. This is the core surface $\text{\ensuremath{\protect\core\left(\left\langle \left[a,b\right]\right\rangle \right)}}$.
Because of symmetry, every morphism $Y\to Z$ looks the same. The
image of $Y$ in such a morphism is a tiled surface $Y'$ consisting
of three vertices, eight edges and one octagon. It is not hard to
see that in this case, although $Y'$ is compact, $\protect\sbr\left(Y'\protect\hookrightarrow Z\right)$
is the entire of $Z$ and, in particular, not compact.}
\end{figure}

We need the following lemma when we construct resolutions in \cite[Section 2.3]{magee2020asymptotic}. 
\begin{lem}
\label{lem:g(SBR(Y)) subset SBR(g(Y))}Let $f\colon Z_{1}\to Z_{2}$
be a morphism between two strongly boundary reduced tiled surfaces
and let $Y$ be a sub-surface of $Z_{1}$. Then 
\[
f\left(\sbr\left(Y\hookrightarrow Z_{1}\right)\right)\subseteq\sbr\left(f\left(Y\right)\hookrightarrow Z_{2}\right).
\]
\end{lem}

\begin{proof}
Denote $W=\sbr\left(f\left(Y\right)\hookrightarrow Z_{2}\right)$.
Let $X=f^{-1}\left(W\right)\cap\sbr\left(Y\hookrightarrow Z_{1}\right)$
be the sub-surface of $\sbr\left(Y\hookrightarrow Z_{1}\right)$ consisting
of all cells which are mapped by $f$ to $W$. Assume towards contradiction
that $X$ is a proper sub-surface of $\sbr\left(Y\hookrightarrow Z_{1}\right)$.
This means that $X$ is not strongly boundary reduced, and so $\partial X$
contains a half-block or a half-chain $b$, and the $4g$-gons along
$b$ belong to $\sbr\left(Y\hookrightarrow Z_{1}\right)$. Let $O$
be one such $4g$-gon. But then $f\left(b\right)\subseteq W$ while
$f\left(O\right)\notin W$. As in the proof of Proposition \ref{prop:(S)BR-closure is (S)BR},
$f\left(O\right)$ lies along a half-block, a long chain or a half-chain
in $\partial W$, which is a contradiction.
\end{proof}
To analyze the properties of (strongly) boundary reduced tiled surfaces,
and later on of core surfaces, we need a result of Birman and Series
that strengthens classical results of Dehn \cite{Dehn}. This result
deals with shortest representatives of conjugacy classes in surface
groups. The paper \cite{BirmanSeries} concerns a wide class of presentations
of Fuchsian group, which includes, in particular, the presentations
$\Gamma_{g}=\left\langle a_{1},b_{1},\ldots,a_{g},b_{g}\,\middle|\,\left[a_{1},b_{1}\right]\cdots\left[a_{g},b_{g}\right]\right\rangle $
for every $g\ge2$ as in  (\ref{eq:Gamma_g}). We state here one of
the main results for this case.

Let $J\le\Gamma_{g}$ be a subgroup. Consider the covering space $\Upsilon=J\backslash\widetilde{\Sigma_{g}}$
corresponding to $J$ (here $\widetilde{\Sigma_{g}}$ is the universal
cover from Example \ref{exa:universal-cover}). This $\Upsilon$ is
a tiled surface without boundary that may be compact or not (depending
on whether $J$ has finite index in $\Gamma_{g}$ or not). Conjugacy
classes in $J$ are in one-to-one correspondence with free homotopy
classes of oriented closed curves in $\Upsilon$, and each such class
has representatives contained in $\Upsilon^{\left(1\right)}$. 

In particular, for an arbitrary $1\ne\gamma\in\Gamma_{g}$, consider
the tiled surface $\Upsilon=\left\langle \gamma\right\rangle \backslash\widetilde{\Sigma_{g}}$.
Topologically, this is a two punctured sphere. The conjugacy class
of $\gamma$ in $\Gamma_{g}$ corresponds to the free-homotopy class
of the essential simple closed curve\footnote{We call a closed curve in a surface \emph{essential} if it is not
null-homotopic.} in $\Upsilon$ (with an appropriate orientation). The set of cyclically
reduced cyclic words in $\left\{ a_{1}^{\pm1},\ldots,b_{g}^{\pm1}\right\} $
representing the conjugacy class of $\gamma$ in $\Upsilon$ is identical
to the set of cyclic words coming from non-backtracking cycles in
$\Upsilon^{\left(1\right)}$ representing the same free-homotopy class
of curves. Given a cycle $\CC$ in $\Upsilon^{\left(1\right)}$, a
``half-block switch'' consists of identifying a half-block in $\CC$
or in $\CC^{*}$, and replacing it with the complement half-block
(around the same $4g$-gon). A ``half-chain switch'' can take place
if one of $\CC$ or $\CC^{*}$ is a half-chain, in which case it refers
to replacing this half-chain with its complement. For example, in
the middle part of Figure \ref{fig:core surfaces - examples}, there
is a cycle $\CC$ reading $aba^{-2}b^{-1}c$, which is a half-chain.
Its complement reads $cd^{-1}c^{-1}a^{-1}dc$.
\begin{thm}[Birman-Series]
\label{thm:birman-series} \cite[Thm. 2.12]{BirmanSeries}
\begin{enumerate}
\item The cyclically reduced cyclic word $w$ in $\left\{ a_{1}^{\pm1},\ldots,b_{g}^{\pm1}\right\} $
is a shortest representative of the conjugacy class in $\Gamma_{g}$
it represents if and only if the corresponding bi-infinite periodic
path $\mathcal{C}$ in $\widetilde{\Sigma_{g}}^{\left(1\right)}$
and its inverse $\mathcal{C}^{*}$ do not contain any long block or
long chain.
\item Assume that the cyclic words $w_{1},w_{2}$ are both shortest representatives
of the conjugacy class $\gamma^{\Gamma_{g}}$ for some $1\ne\gamma\in\Gamma_{g}$.
Let $\mathcal{C}_{1}$ and $\mathcal{C}_{2}$ be corresponding cycles
in $\Upsilon=\left\langle \gamma\right\rangle \backslash\widetilde{\Sigma_{g}}$.
Then either $\mathcal{C}_{2}$ can be obtained from $\mathcal{C}_{1}$
by a finite number of half-block switches, or $\mathcal{C}_{2}$ can
be obtained from $\CC_{1}$ by a single half-chain switch. 
\end{enumerate}
\end{thm}

For example, the element $\gamma=aba^{-2}b^{-1}c\in\Gamma_{2}$ has
exactly two different cyclic words which are shortest representatives
of its conjugacy class: the cyclic words $aba^{-2}b^{-1}c$ and $cd^{-1}c^{-1}a^{-1}dc$.
These two words correspond to two disjoint cycles in the $1$-skeleton
of $\left\langle \gamma\right\rangle \backslash\widetilde{\Sigma_{2}}$,
and the complement of their union consists of three components: two
with infinitely many $4g$-gons, and one component, an annulus bounded
by both cycles, containing exactly two octagons.
\begin{cor}
\label{cor:no long blocks/chains =00003D=00003D> cyclically shortest}Let
$Y$ be a tiled surface and let $\CC$ be a (non-backtracking) cycle
in $Y^{\left(1\right)}$. If $\CC$ and $\CC^{*}$ do not contain
any long block or long chain, then $\CC$ is a shortest representative
of its free-homotopy class in $Y$. 
\end{cor}

\begin{proof}
If $\CC$ and $\CC^{*}$ contain no long block nor long chain, then,
by Theorem \ref{thm:birman-series}, they are shortest representatives
of $w\left(\CC\right)^{\Gamma_{g}}$, the conjugacy class of $w\left(\CC\right)$
in $\Gamma_{g}$. Every other cycle representing the same free-homotopy
class in $Y$, also represents $w\left(\CC\right)^{\Gamma}$, so it
cannot be shorter.
\end{proof}
If $Y$ is boundary reduced, the converse also hold.
\begin{cor}
\label{cor:cycles in BR}Let $Y$ be a boundary reduced tiled surface,
and let $\CC$ be a (non-backtracking) cycle in $Y^{\left(1\right)}$.
Then the following holds.
\begin{enumerate}
\item \label{enu:BR =00003D=00003D> exists shortest representative of every cycle}There
is a shortest representative cycle in $Y^{\left(1\right)}$ for the
free homotopy class of $\CC$. 
\item \label{enu:in BR cyclically shortest =00003D=00003D> no long blocks and chains}If
$\CC$ is a shortest representative of its free-homotopy class in
$Y$ then $\CC$ and $\CC^{*}$ do not contain any long block or long
chain.
\end{enumerate}
\end{cor}

\begin{proof}
As $Y$ is boundary reduced, every long block contained in $Y^{\left(1\right)}$
lies at the boundary of a $4g$-gon contained in $Y$, and therefore
so does the complement of this long block. Similarly, if $\PP$ is
a long chain contained in $Y^{\left(1\right)}$, then every sequence
of $4g$-gons along $\PP$ which are not in $Y$ gives rise to a long
block or a long chain along $\partial Y$, which is impossible. Hence
all $4g$-gon along $\PP$ belong to $Y$, and so the complement of
$\PP$ belongs to $Y$.

For an arbitrary $\CC\subseteq Y^{\left(1\right)}$, we can greedily
shorten it by replacing every long block or long chain in $\CC$ or
in $\CC^{*}$ by their complement, and this leads to a shortest representative
by Theorem \ref{thm:birman-series}. This proves (\ref{enu:BR =00003D=00003D> exists shortest representative of every cycle}).
Now assume that $\CC\subseteq Y^{\left(1\right)}$ is cyclically shortest.
If $\CC$ or $\CC^{*}$ contains a long block or a long chain, then
replacing this long block/chain with its complement, which is still
in $Y$, reduces the length of $\CC$. This proves (\ref{enu:in BR cyclically shortest =00003D=00003D> no long blocks and chains}). 
\end{proof}
\begin{cor}
\label{cor:BR subsurface is pi1-injective}Every morphism from a boundary
reduced tiled surface is $\pi_{1}$-injective.
\end{cor}

\begin{proof}
Let $f\colon Y\to Z$ be a morphism of tiled surfaces where $Y$ is
boundary reduced, and let $\CC$ be a cycle in $Y$ which is not null-homotopic.
By Corollary \ref{cor:cycles in BR} there is a representative in
$Y^{\left(1\right)}$ for the same free-homotopy class in $Y$, which
contains no long blocks nor long chains. But then the $f$-image of
this representative is also shortest, and in particular non-nullhomotopic,
in $Z$.
\end{proof}

\section{Properties and construction of core surfaces\label{sec:core surfaces}}

\subsection{Properties of core surfaces\label{subsec:Properties of core surfaces}}

Recall Definition \ref{def:core-surface}: the core surface $\core\left(J\right)$
of a subgroup $J\le\Gamma_{g}$ is the sub-complex of $\Upsilon=J\backslash\widetilde{\Sigma_{g}}$
obtained as the union of all shortest representative cycles in $\Upsilon^{\left(1\right)}$
of non-trivial conjugacy classes of $J$, together with the connected
components of the complement which contain finitely many $4g$-gons.
In this section we prove some basic properties of this object. Among
these, we show that a core surface is strongly boundary reduced, that
it is compact whenever $J$ is f.g., and that whenever $H\le J\le\Gamma_{g},$
the natural morphism $H\backslash\widetilde{\Sigma_{g}}\to J\backslash\widetilde{\Sigma_{g}}$,
restricts to a map between the corresponding core surfaces. 

We start with analyzing the special case of the core surface of a
cyclic subgroup.
\begin{lem}
\label{lem:cyclic core surfaces}Let $1\ne\gamma\in\Gamma_{g}$ be
a non-trivial element. Then the core surface $\core\left(\left\langle \gamma\right\rangle \right)$
is connected and compact with its thick version homeomorphic to an
annulus. Furthermore, both boundary cycles of $\core\left(\left\langle \gamma\right\rangle \right)$
are of length $\ell\left(\gamma^{\Gamma_{g}}\right)$.
\end{lem}

\begin{proof}
That $\core\left(\left\langle \gamma\right\rangle \right)$ is connected
follows immediately from Theorem \ref{thm:birman-series}. Let $\mathcal{C}$
be some (simple) cycle in the twice-punctured sphere $\Upsilon=\left\langle \gamma\right\rangle \backslash\widetilde{\Sigma_{g}}$
which is a shortest representative of $\gamma^{\Gamma_{g}}$. If $\mathcal{C}$
or $\mathcal{C}^{*}$ is a half-chain, denote by $\CC'$ its complement.
Then $\core\left(\left\langle \gamma\right\rangle \right)$ is precisely
the compact annulus made of $\CC$, $\CC'$ and the narrow annulus
separating them. Indeed, none of $\CC$, $\CC'$ and their inverses
contain a half-block, so by Theorem \ref{thm:birman-series}, $\CC$
and $\CC'$ are the only shortest representatives of $\gamma^{\Gamma_{g}}$.

Now assume neither $\CC$ nor $\CC^{*}$ is a half-chain. By Theorem
\ref{thm:birman-series}, $\core\left(\left\langle \gamma\right\rangle \right)$
may be obtained from the tiled sub-surface $\CC$ of $\Upsilon$ by
repeatedly annexing any $4g$-gon sitting along a half-block at the
boundary. We only need to show this process must end after finitely
many steps.

Let $Y_{0}=\CC,Y_{1},Y_{2},\ldots$ denote the sub-surfaces we construct
in this process, so $Y_{i+1}$ is obtained from $Y_{i}$ by annexing
a $4g$-gon of $\Upsilon$ bordering a half-block in $\partial Y_{i}$.
For every $i$, as $\mathcal{C}$ is a shortest representative, the
boundary component of $Y_{i}$ around each of the two punctures of
$\Upsilon$ is of length at least $\left|\CC\right|$. Clearly, $\left|\partial Y_{i+1}\right|\le\left|\partial Y_{i}\right|$,
so we get that every $Y_{i}$ has exactly two boundary components,
and each of them is of length $\left|\CC\right|$. In particular,
the $4g$-gon annexed to $Y_{i}$ to obtain $Y_{i+1}$ contains at
its boundary a half block ${\cal B}$ which did not belong to $Y_{i}$,
and such that ${\cal B}^{*}$ is a path in $\partial Y_{i+1}$.

Finally, denote by $s$ the number of hanging half-edges at the boundary
of $\left(Y_{0}\right)_{+}$. At every step, the boundary of $Y_{i}$
has constant length $2\left|\CC\right|$, while $s$ increases by
$\left(2g-1\right)\left(4g-2\right)-2=8g\left(g-1\right)$ --- see
Figure \ref{fig:half block switch}. As the number of hanging half-edges
in a compact tiled surface $Y$ without isolated vertices nor leaves
is at most $\left(4g-2\right)\left|\partial Y\right|$, this process
must terminate after finitely many steps.
\end{proof}
\begin{figure}
\begin{centering}
\includegraphics{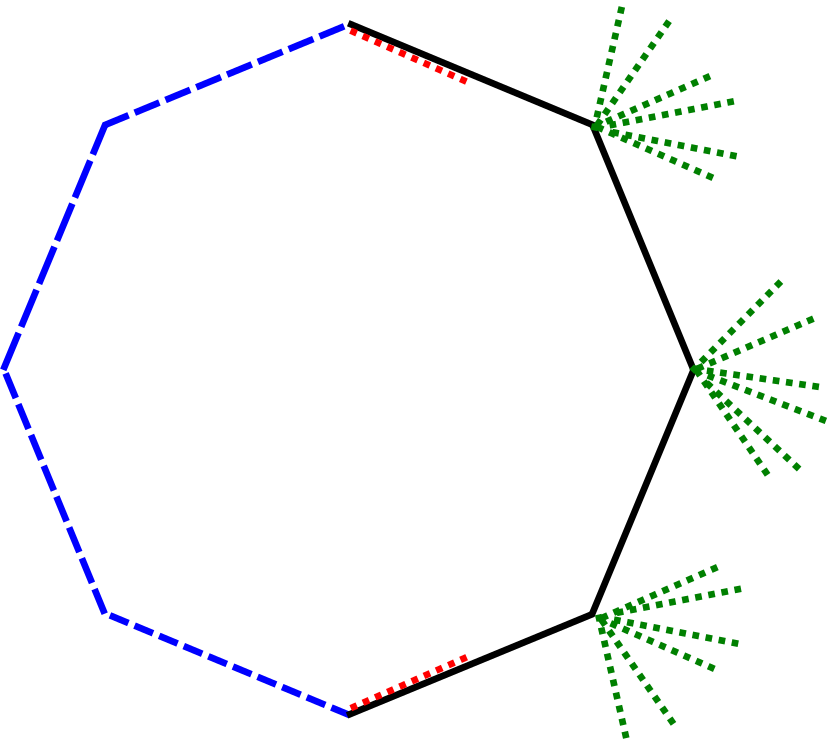}
\par\end{centering}
\caption{\label{fig:half block switch}Let $g=2$. Assume $Y$ is a tiled sub-surface
of a tiled surface $Z$ with no boundary. Assume that $Y$ has some
half-block at its boundary, marked here in broken blue, and sitting
along the octagon $O$. Denote by $Y'$ the union of $Y$ with the
closure of $O$ in $Z$, and assume that the inverse of the other
half block along $O$ is an interval along $\partial Y'$. Then, the
number of hanging half-edges in $\left(Y'\right)_{+}$ is larger by
$16$ than their number in $Y_{+}$: two hanging half-edges of $Y_{+}$
(marked in red) are no longer hanging in $\left(Y'\right)_{+}$, but
$18=\left(2g-1\right)\left(4g-2\right)$ new hanging half-edges (marked
in green) belong to $\left(Y'\right)_{+}$.}
\end{figure}

As $\core\left(J\right)$ is a closed sub-surface of $\Upsilon=J\backslash\tsg$,
every component of its complement $\Upsilon\setminus\core\left(J\right)$
is open. Hence every component is a surface with punctures. Each of
these punctures corresponds either to a puncture of $\Upsilon$, in
which case we call it a \emph{funnel}, or to a component of $\partial\core\left(J\right)$,
in which case we call it a \emph{fake-puncture}. In particular, a
funnel is of infinite distance (measured in paths of adjacent $4g$-gons,
say) from any given $4g$-gon in $\Upsilon\setminus\core\left(J\right)$,
while a fake-puncture has certain (open) $4g$-gons adjacent to it. 
\begin{lem}
\label{lem:If proper surface has g>0 or >3 punctures, it has a shortest element not contained in boundary}Let
$X$ be a connected tiled surface, which is a proper surface (so every
edge is incident with one or two $4g$-gons, and every vertex is incident
with a single sequence of $m$ $4g$-gons, $1\le m\le4g$). Assume
that $X$ has genus $g$ and a total of $b$ boundary components and
punctures. If $g\ge1$ or $b\ge3$, then $X^{\left(1\right)}$ contains
a cycle, not contained in $\partial X$, which is cyclically shortest
among the cycles representing its free homotopy class in $X$.
\end{lem}

\begin{proof}
First assume the genus of $X$ is positive. Then it contains two non-homotopic
non-separating (and thus essential) simple closed curves $\alpha$
and $\beta$ away from its boundary with intersection number one (not
necessarily contained in $X^{\left(1\right)}$). Because every representative
of the free homotopy class $\left[\alpha\right]$ should intersect
$\beta$, we get that the shortest representative of this class is
not contained in $\partial X$.

Now assume that $X$ is a sphere with $b\ge3$ boundary components
and/or punctures. By the assumption that $X$ is a proper surface,
$\partial X$ consists of disjoint connected components, each of which
homeomorphic to $S^{1}$. A figure-eight curve around two of the punctures/boundary
components is not homotopic in $X$ to any power of a loop around
one of the boundary components. This proves the lemma.
\end{proof}
\begin{prop}[Basic properties of core surfaces]
\label{prop:properties of core surfaces} Let $J$ be a non-trivial
subgroup of $\Gamma$ and let $\Upsilon\eqdf J\backslash\tsg$ be
the corresponding covering space of $\Sigma_{g}$. Then the following
properties hold.
\begin{enumerate}
\item \label{enu:every boundary cycle of a core surface is essential}Every
boundary cycle $\delta$ of $\core\left(J\right)$ is an essential
curve of $\Upsilon$.
\item \label{enu:boundary cycles of core surfaces represent shortest words}For
every boundary cycle $\delta$ of $\core\left(J\right)$, $w(\delta)$
is a cyclically shortest word.
\item \label{enu:core surfaces are strongly boundary reduced}$\core\left(J\right)$
is strongly boundary reduced.
\item \label{enu:core(j) obtained by trimming funnels}Every connected component
of the complement of $\core\left(J\right)$ in $\Upsilon$ is homeomorphic
to a twice-punctured sphere, with one funnel and one fake-puncture.
In particular, $\core\left(J\right)$ is a deformation retract of
$\Upsilon=J\backslash\tsg$.
\item \label{enu:core surfaces are connected}$\core\left(J\right)$ is
connected.
\item \label{enu:pi1 of core}The embedding $\core\left(J\right)\hookrightarrow\Upsilon$
induces an isomorphism in the level of fundamental groups.
\item \label{enu:step(ii) contains only open discs or annuli}In step $\left(ii\right)$
of Definition \ref{def:core-surface}, the connected components with
finitely many $4g$-gons that are added to $\core\left(J\right)$
are either open discs or open annuli (twice-punctured spheres).
\end{enumerate}
\end{prop}

\begin{proof}
Let $\delta$ be a boundary cycle of $\core\left(J\right)$. If $\delta$
is null-homotopic in $\Upsilon$, then it bounds a disc in one of
its sides. This side cannot be external to $\core\left(J\right)$,
because then it should have been annexed to $\core\left(J\right)$
by part $\left(ii\right)$ of Definition \ref{def:core-surface}.
If the disc is on the internal side of $\delta$, then the connected
component of $\delta$ in $\core\left(J\right)$ does not contain
any essential curve of $\Upsilon$. This is impossible. Hence $\delta$
is essential and (\ref{enu:every boundary cycle of a core surface is essential})
is proved.

For (\ref{enu:boundary cycles of core surfaces represent shortest words}),
by Corollary \ref{cor:no long blocks/chains =00003D=00003D> cyclically shortest},
it is enough to show that $\delta$ and $\delta^{*}$ contain no long
blocks nor long chains. We begin with $\delta^{*}$. Suppose that
$\delta^{*}$ contains a long block $b$, and let $\overline{b}$
denote its complement (along the same $4g$-gon $O_{b}$ of $\Upsilon$).
Consider the $1$-skeleton of $\core\left(J\right)$. All the internal
vertices in $b$ (vertices contained in $b$ but not at its endpoints)
have degree two in $\core\left(J\right)^{\left(1\right)}$, so any
non-backtracking cycle traversing one edge of $b$ must traverse all
of $b$, and can be shortened in $\Upsilon$ by replacing $b$ with
$\overline{b}$. So there is no shortest cyclic representative using
any edge of $b$, and after step $\left(i\right)$ of Definition \ref{def:core-surface},
$O_{b}$ belongs to the same connected component of $\Upsilon\backslash\core\left(J\right)$
as the $4g$-gons on the other side of $b$. But in step $\left(ii\right)$
of Definition \ref{def:core-surface}, $O_{b}$ can only be annexed
to $\core\left(J\right)$ if the entire component is, which is not
the case. This is a contradiction. A similar argument shows that $\delta^{*}$
cannot contain any long chain $c$: indeed, any non-backtracking cycle
in the 1-skeleton of $\core\left(J\right)$ that intersects the interior
of $c$ must contain a long block or a long chain, so no shortest
cyclic representative intersects the interior of $c$. Hence $\delta^{*}$
contains no long block nor long chain.

We still need to show that $\delta$ contains neither long blocks
nor long chains. Denote by $\mathbb{CORE}\left(J\right)$ a realization
of the thick version of $\core\left(J\right)$ in $\Upsilon$. In
particular, replacing $\core\left(J\right)$ with $\mathbb{CORE}\left(J\right)$
does not alter the topology of the complement $\Upsilon\backslash\core\left(J\right)$.
Let $C$ be the connected component of $\Upsilon\backslash\mathbb{CORE}\left(J\right)$
bordering $\delta$, and let $\overline{C}$ denote the closure of
$C$ in $\Upsilon$. So the difference between $C$ and $\overline{C}$
is that every fake-puncture of $C$ becomes a closed connected component
of $\partial\overline{C}$. We think of $\overline{C}$ as a tiled
surface. Formally, every vertex or edge of $\Upsilon$ that belongs
to two (or more) different boundary pieces $\overline{C}$, is duplicated
in $\overline{C}$. Now $\overline{C}$ is a 2-complex with directed
and labeled edges which satisfies properties \textbf{P1} and \textbf{P3}
of Proposition \ref{prop:comb-def-of-tiled-surface}. Because for
every boundary cycle $\delta$ of $\core\left(J\right)$ we have that
$\delta^{*}$ contains neither long blocks nor long chains, we deduce
that $\overline{C}$ is boundary reduced. It now follows from Proposition
\ref{prop:BR core graph with octagons is a tiled surface} that $\overline{C}$
is indeed a legitimate tiled surface, and a boundary reduced one.
By Corollary \ref{cor:cycles in BR}, as $\delta\subseteq\overline{C}$,
we get that $\delta$ can be shortened to a shortest cyclic representative
$\delta'\subseteq\overline{C}$ of its free-homotopy class in $\Upsilon$.

Now suppose that $\delta$ contains a long block or a long chain.
Then $\delta'$ is different from $\delta$, and it is in $\core\left(J\right)$
by definition. So $\delta'\subseteq\overline{C}\cap\core\left(J\right)$.
Because $\delta$ and $\delta'$ are isotopic, different, and lie
in $\partial\overline{C}$, we get that $C$ must be a two-punctured
sphere containing finitely many $4g$-gons, and therefore should have
been part of $\core\left(J\right)$ by part $\left(ii\right)$ of
Definition \ref{def:core-surface}. This is a contradiction, and (\ref{enu:boundary cycles of core surfaces represent shortest words})
is proven.

If the boundary component $\delta$ of $\core\left(J\right)$ contains
a half-block, then a half-block switch yields another shortest representative
and should be in $\core\left(J\right)$ together with the $4g$-gon
along which the half-block lies. A similar argument works if $\delta$
is a half-chain. This proves (\ref{enu:core surfaces are strongly boundary reduced}).

Let $C$ be (again) a connected component of $\Upsilon\setminus\mathbb{CORE}\left(J\right)$.
As $C$ is open, it is a surface and so $\overline{C}$ is a boundary
reduced tiled surface which is a proper surface. Because $\overline{C}$
is boundary reduced, every free homotopy class of curves in $\overline{C}$
has a cyclically shortest representatives in $C$ without long blocks
or long chains (Corollary \ref{cor:cycles in BR}). But such a cycle
also belongs to $\core\left(J\right)$ by the definition of a core
surface, and so is contained in $\partial\overline{C}$. By Lemma
\ref{lem:If proper surface has g>0 or >3 punctures, it has a shortest element not contained in boundary},
$C$ must be a sphere with at most two punctures. By the fact that
$\Upsilon$ is connected and by part $\left(ii\right)$ of Definition
\ref{def:core-surface}, $C$ must have two punctures: one fake and
one a funnel. This settles item (\ref{enu:core(j) obtained by trimming funnels}).
Items (\ref{enu:core surfaces are connected}) and (\ref{enu:pi1 of core})
follow immediately. 

Finally, let $Y$ denote the union of cyclically shortest representatives
in $\Upsilon$ as in the first step of Definition \ref{def:core-surface}.
Let $C$ be a connected component of the complement of the thick version
of $Y$ in $\Upsilon$ with finitely many $4g$-gons. As $C$ is open,
it is a surface, and therefore $\overline{C}$ is a proper surface.
As above, we think of $\overline{C}$ as a $2$-complex (while duplicating
vertices and edges of $\Upsilon$ appearing in different boundary
pieces of $\overline{C}$), its boundary is reduced because $Y$ is
made of shortest cycles only, and therefore $\overline{C}$ is a tiled
surface by Proposition \ref{prop:BR core graph with octagons is a tiled surface}.
By Lemma \ref{lem:If proper surface has g>0 or >3 punctures, it has a shortest element not contained in boundary},
unless $\overline{C}$ is a disc or an annulus, it contains a cyclically
shortest cycle not contained in $\partial\overline{C}$. This is a
contradiction to the definition of $Y$, and item (\ref{enu:step(ii) contains only open discs or annuli})
is proven.
\end{proof}
Consider two subgroups $J_{1},J_{2}\le\Gamma_{g}$. It follows from
a standard fact in the theory of covering spaces that there is a morphism
of tiled surfaces $J_{1}\backslash\tsg\to J_{2}\backslash\tsg$ commuting
with the quotient maps from $\tsg$, if and only if $J_{1}^{~\gamma}\le J_{2}$
for some conjugate $J_{1}^{~\gamma}=\gamma J_{1}\gamma^{-1}$ of $J_{1}$.
In this case, any morphism $J_{1}\backslash\tsg\to J_{2}\backslash\tsg$
restricts to a morphism of the corresponding core surfaces:
\begin{lem}
\label{lem:morphisms of core surfaces}Let $J_{1}\le J_{2}\le\Gamma_{g}$
and let $f\colon J_{1}\backslash\tsg\to J_{2}\backslash\tsg$ be the
natural morphism. Then,
\begin{enumerate}
\item \label{enu:f restrict to a morphism of core surfaces}$f$ restricts
to a morphism $\core\left(J_{1}\right)\to\core\left(J_{2}\right)$,
and
\item \label{enu:shortest cycles come from shortest cycles}for $1\ne\gamma\in J_{1}$,
every shortest representative cycle of $\gamma^{J_{2}}$ in $\core\left(J_{2}\right)$
is an $f$-image of a shortest representative cycle of $\gamma^{J_{1}}$
in $\core\left(J_{1}\right)$.
\end{enumerate}
\end{lem}

\begin{proof}
For $i=1,2$, denote $\Upsilon_{i}=J_{i}\backslash\tsg$. By definition,
the morphism $f$ preserves the orientation and labels of edges. So
it follows from Corollaries \ref{cor:no long blocks/chains =00003D=00003D> cyclically shortest}
and \ref{cor:cycles in BR} that it maps shortest representative cycles
of free-homotopy classes to shortest representative cycles of free-homotopy
classes. We now show that the connected components we add to the core
surface of $J_{1}$ in part $\left(ii\right)$ of Definition \ref{def:core-surface}
are also mapped to the core surface of $J_{2}$.

Let $T$ be such a connected component, namely, a connected component
of the complement of the union of shortest cycles in $\Upsilon_{1}$
which is added to $\core\left(J_{1}\right)$ in part $\left(ii\right)$
of Definition \ref{def:core-surface}. Consider a $4g$-gon $O$ in
$T$. Let $T'$ denote the connected component of $f\left(O\right)$
in the complement of the union of shortest cycles in $\Upsilon_{2}$.
We claim that $T'$ contains finitely many $4g$-gons, and therefore,
by Definition \ref{def:core-surface}, must be contained in $\core\left(J_{2}\right)$.
In fact, all the $4g$-gons in $T'$ are images of $4g$-gons in $T$,
and therefore there are finitely many of them. Indeed, for every $4g$-gon
$O'$ in $T'$, there is a ``path of $4g$-gons'' inside $T'$ from
$f\left(O\right)$ to $O'$, where each $4g$-gon shares an edge with
the previous one. If we lift this path to a path of $4g$-gons from
$O$ in $\Upsilon_{1}$, it cannot leave the connected component $T$.
Hence $O'$ is an image of some $4g$-gon in $T$. This proves (\ref{enu:f restrict to a morphism of core surfaces}).

Now let $1\ne\gamma\in J_{1}$ and let $\CC$ be a shortest representative
cycle of $\gamma^{J_{1}}$ in $\Upsilon_{1}$, so its image $f\left(\CC\right)$
is a shortest representative for $\gamma^{J_{2}}$ in $\Upsilon_{2}$,
as noted above. For any other cycle $\CC'$ which is a shortest representative
of $\gamma^{J_{2}}$ in $\Upsilon_{2}$, the (free) homotopy between
$f\left(\CC\right)$ and $\CC'$ in $\Upsilon_{2}$ can be lifted
to $\Upsilon_{1}$ (this follows from the general theory of covering
spaces, e.g., \cite[Page 30]{hatcher2005algebraic}) and therefore,
in particular, $\CC'$ is an $f$-image of a cycle in $\Upsilon_{1}$
which represents $\gamma^{J_{1}}$ and of the same length as $\CC$.
This proves (\ref{enu:shortest cycles come from shortest cycles}).
\end{proof}
\begin{lem}
\label{lem:SBR =00003D=00003D> every shortest representative of a cycle is contained}Let
$Y$ be a strongly boundary reduced tiled surface embedded in $\Upsilon=J\backslash\tsg$
and let $\CC\subseteq Y^{\left(1\right)}$ be a non-nullhomotopic
cycle. Then \emph{every} shortest representative in $\Upsilon$ of
the free homotopy class of $\CC$ is contained in $Y$.
\end{lem}

\begin{proof}
By Corollary \ref{cor:cycles in BR}(\ref{enu:BR =00003D=00003D> exists shortest representative of every cycle}),
there is a shortest representative of the free homotopy class of $\CC$
in $Y$, and without loss of generality, assume that $\CC$ is shortest.
Further assume that $\CC$ represents the conjugacy class $\gamma^{J}$
in $J$, and consider the morphism $f\colon\left\langle \gamma\right\rangle \backslash\tsg\to\Upsilon$.
Lemma \ref{lem:morphisms of core surfaces}(\ref{enu:shortest cycles come from shortest cycles})
shows that every shortest representative of $\gamma^{J}$ in $\Upsilon$
is an $f$-image of a shortest representative in $\core\left(\left\langle \gamma\right\rangle \right)$,
and Lemma \ref{lem:cyclic core surfaces} shows that are finitely
many such representatives in $\core\left(\left\langle \gamma\right\rangle \right)$.
As the $f$-image of a half block (a half-chain) in $\core\left(\left\langle \gamma\right\rangle \right)$
is a half-block (half-chain respectively) in $\Upsilon$, we get that
all shortest representatives of $\gamma^{J}$ are obtained from $\CC$
by half-block switches or a half-chain switch. As $Y$ is strongly
boundary reduced, it contains the complement of every half-block or
half-chain in it. 
\end{proof}
\begin{lem}
\label{lem:BR containing J has complements two-punctures spheres}Let
$Y$ be a connected boundary reduced tiled surface embedded in $\Upsilon=J\backslash\tsg$
such that $p_{*}\left(\pi_{1}\left(Y,y\right)\right)=J\le\Gamma_{g}$.
Then every component of $\Upsilon\setminus Y$ is a twice-punctured
sphere with one funnel and one fake-puncture.
\end{lem}

\begin{proof}
As $p_{*}\left(\pi_{1}\left(Y,y\right)\right)=J$, $Y$ contains a
representative of every free homotopy class in $\Upsilon$. Let $C$
be a connected component of $\Upsilon\setminus Y$. As in the proof
of Lemma \ref{lem:If proper surface has g>0 or >3 punctures, it has a shortest element not contained in boundary},
$C$ cannot have positive genus. It cannot be a once-punctured sphere
because if this puncture is a funnel, $\Upsilon$ is not connected,
and if it is a fake-puncture, the boundary component of $Y$ along
this fake-puncture is not boundary reduced (by \cite{Dehn}). In addition,
$C$ cannot have two fake-punctures, because $Y$ is connected and
so there would be a free homotopy class of essential curves such that
any representative must go through $C$ (between these two punctures).
If $C$ contains a funnel, then any cycle in $C$ representing a loop
around this funnel has a shortest representative in $Y$. These two
cycles are isotopic in $\Upsilon$ and therefore bound an annulus.
So $C$ cannot contain two different funnels. We conclude that $C$
is a twice-punctured sphere with one funnel and one fake-puncture.
\end{proof}
\begin{lem}
\label{lem:SBR containing J contains core(J)}Let $Y$ be a strongly
boundary reduced tiled surface embedded in $\Upsilon=J\backslash\tsg$
such that $p_{*}\left(\pi_{1}\left(Y,y\right)\right)=J\le\Gamma_{g}$.
Then $Y\supseteq\core\left(J\right)$.
\end{lem}

\begin{proof}
By Lemma \ref{lem:SBR =00003D=00003D> every shortest representative of a cycle is contained},
$Y$ contains every shortest representative of every non-trivial free
homotopy class, and so contains the subcomplex from part $\left(i\right)$
of Definition \ref{def:core-surface}. By Lemma \ref{lem:BR containing J has complements two-punctures spheres},
every component of the complement of $Y$ contains a funnel and thus
infinitely many $4g$-gons. In particular, it cannot be contained
in one of the components added to $\core\left(J\right)$ in part $\left(ii\right)$
of Definition \ref{def:core-surface}. This completes the proof.
\end{proof}
\begin{prop}
\label{prop:core surfaces of fg groups are compact}If $J\le\Gamma$
is finitely generated then $\core\left(J\right)$ is compact.
\end{prop}

\begin{proof}
Suppose that $J\le\Gamma$ is finitely generated and let $S=\left\{ w_{1},\ldots,w_{k}\right\} $
be a finite generating set represented as words in $\left\{ a_{1}^{\pm1},\ldots,b_{g}^{\pm1}\right\} $.
Let $\left(\Upsilon,q\right)=J\backslash\left(\tsg,u\right)$ be the
pointed quotient of $\tsg$ with the base point $q$ being the image
of the base point $u$. Then $w_{1},\ldots,w_{k}$ correspond to unique,
possibly not cyclically reduced, cycles $\CC_{1},\ldots,\CC_{k}$
based at $q$. Consider the sub-surface $Y$ of $\Upsilon$ consisting
of the union $\bigcup_{i=1}^{k}\CC_{i}$, and let $Y'=\br\left(Y\hookrightarrow\Upsilon\right)$.
By Propositions \ref{prop:(S)BR-closure is (S)BR} and \ref{prop:BR closure of compact is compact},
$Y'$ is a compact boundary reduced tiled surface containing $Y$.

Finally, enlarge $Y'$ to obtain a slightly larger sub-surface $Y''\subseteq\Upsilon$
by repeatedly adding any $4g$-gon which borders some half-block or
half-chain in $\partial Y'$. As in the proof of Lemma \ref{lem:SBR =00003D=00003D> every shortest representative of a cycle is contained},
the $4g$-gons added next to a boundary cycle representing $\gamma^{J}$
are all images of the finitely many $4g$-gons in $\core\left(\left\langle \gamma\right\rangle \right)$,
and so there are finitely many steps near this boundary component.
As $Y'$ is compact, it has finitely many boundary components, and
hence $Y''$ is constructed in finitely many step and is compact too.
Moreover, by the way it is constructed, $Y''$ is strongly boundary
reduced. By Lemma \ref{lem:SBR containing J contains core(J)}, the
compact $Y''$ contains $\core\left(J\right)$. This proves the proposition.
\end{proof}
We can now give an intrinsic definition for a core-surface, not relying
on a given subgroup of $\Gamma_{g}$.
\begin{prop}[Intrinsic definition of a core surface]
\label{prop:intrinsic def of core surface} A (non-empty) tiled surface
is a core surface if and only if it is $\left(i\right)$ connected,
$\left(ii\right)$ strongly boundary reduced, $\left(iii\right)$
every boundary cycle is a cyclically shortest representative of its
free homotopy class, and $\left(iv\right)$ it contains no funnels.
\end{prop}

Note that if $Y$ is a \emph{compact }tiled surface, then condition
$\left(iv\right)$ in the proposition automatically holds.
\begin{proof}
That a core surface satisfies properties $\left(i\right)$, $\left(ii\right)$
and $\left(iii\right)$ is the content of items (\ref{enu:core surfaces are connected}),
(\ref{enu:core surfaces are strongly boundary reduced}) and (\ref{enu:boundary cycles of core surfaces represent shortest words}),
respectively, in Proposition \ref{prop:properties of core surfaces}.
Now let $Y=\core\left(J\right)\subseteq\Upsilon=J\backslash\tsg$
for some $J\le\Gamma_{g}$. Let $p$ be a funnel-puncture in $\Upsilon$
and let $\CC\subseteq\Upsilon^{\left(1\right)}$ be a cyclically shortest
representative of a simple closed curve around $p$. As in the proof
of Lemma \ref{lem:SBR =00003D=00003D> every shortest representative of a cycle is contained},
there are finitely many shortest representatives of the free homotopy
class of $\CC$, and thus the connected component $C$ of $p$ in
the complement of these representatives in $\Upsilon$, contains infinitely
many $4g$-gons. Moreover, $\Upsilon\setminus C$ is strongly boundary
reduced and a retract of $\Upsilon$, and therefore contains every
shortest representative of free homotopy classes in $\Upsilon$. By
Definition \ref{def:core-surface}, $C$ cannot belong to $\core\left(J\right)$
and thus $\left(iv\right)$ holds.

Conversely, assume that $Y$ is a tiled surface satisfying these four
assumptions. Let $p\colon Y\to\Sigma_{g}$ be the immersion. Choose
some vertex $y\in Y$ and let $J=p_{*}\left(\pi_{1}\left(Y,y\right)\right)\le\Gamma_{g}$.
As in the proof of Proposition \ref{prop:comb-def-of-tiled-surface},
there is an embedding $r\colon Y\hookrightarrow\Upsilon=J\backslash\tsg$,
so we may think of $Y$ as a sub-complex of $\Upsilon$. Since $Y$
has the same fundamental group as $\Upsilon$ and is strongly boundary
reduced, it contains every shortest representatives of free homotopy
classes in $\Upsilon$, by Lemma \ref{lem:SBR =00003D=00003D> every shortest representative of a cycle is contained}.
As in the proof of Proposition \ref{prop:core surfaces of fg groups are compact},
every connected component of $\Upsilon\setminus Y$ is a twice-punctured
sphere with one funnel and one fake-puncture and, in particular, contains
infinitely many $4g$-gons. Hence $Y\supseteq\core\left(J\right)$. 

Finally, by Proposition \ref{prop:properties of core surfaces}(\ref{enu:core(j) obtained by trimming funnels}),
every connected component $C$ of $\Upsilon\setminus\core\left(J\right)$
is a twice-punctured sphere, with one funnel and one fake-puncture.
By $\left(iv\right)$, $Y$ does not contain the whole of $C$, and
by $\left(iii\right)$, $Y$ does not contain any point of $C$. We
conclude that $Y=\core\left(J\right)$.
\end{proof}
Using the intrinsic definition of core surfaces from Proposition \ref{prop:intrinsic def of core surface},
we conclude that we have a one-to-one bijection
\begin{equation}
\left\{ {\mathrm{conjugacy~classes~of}\atop \mathrm{subgroups~of}~\Gamma_{g}}\right\} \:\longleftrightarrow\:\left\{ {\mathrm{core~surfaces}\atop \mathrm{labeled~by}~\left\{ a_{1},\ldots,b_{g}\right\} }\right\} \label{eq:1-1 correspondence}
\end{equation}
which restricts to a one-to-one correspondence
\begin{equation}
\left\{ {\mathrm{conjugacy~classes~of}\atop \mathrm{f.g.~subgroups~of}~\Gamma_{g}}\right\} \:\longleftrightarrow\:\left\{ {\mathrm{compact~core~surfaces}\atop \mathrm{labeled~by}~\left\{ a_{1},\ldots,b_{g}\right\} }\right\} .\label{eq:1-1 correspondence of f.g.}
\end{equation}

\subsection{Foldings and construction of core surfaces\label{subsec:Foldings-and-construction}}

One of the most useful concepts introduced in \cite{stallings1983topology}
is that of ``foldings'', now known as Stallings foldings. In graphs,
a folding is a process in which one merges two equally-labeled oriented
edges with the same head-vertex or with the same tail-vertex. Occasionally,
one also trims leaves from the graph. This process allows one to construct
the core graph of a f.g.~subgroup $H$ of the free group $\F_{r}$
from a finite set $\left\{ w_{1},\ldots,w_{k}\right\} $ of generators
as follows: create a bouquet with $k$ petals, one for every generator.
Then fold until no more folding steps are possible. The resulting
graph is the core graph of $H$ (e.g.~\cite[Proposition 3.8]{kapovich2002stallings}).

We now present an analogous folding process for f.g.~subgroups of
surface groups and their core surfaces. 
\begin{thm}[Foldings]
\label{thm:folding} Let $J\le\Gamma_{g}$ be a f.g.~subgroup and
let $\left\{ w_{1},\ldots,w_{k}\right\} $ be a generating set consisting
of words in $\left\{ a_{1}^{\pm1},b_{1}^{\pm1},\ldots,a_{g}^{\pm1},b_{g}^{\pm1}\right\} $.
Then $\core\left(J\right)$ can be constructed via the following finite
process:
\begin{enumerate}
\item \textbf{Preparation: }Without loss of generality, we may assume all
words $w_{1},\ldots,w_{k}$ represent non-trivial elements in $\Gamma_{g}$
(this can be easily and efficiently checked using Dehn's algorithm,
and trivial elements may be removed from the set).
\begin{itemize}
\item Consider a cycle $\CC$ representing $w_{1}$ and shorten it \emph{cyclically}
until one obtains a shortest representative of $w_{1}^{\Gamma_{g}}$
(as in Theorem \ref{thm:birman-series}). Assume the new cyclically
shortest word represents the same element as $w_{1}^{~s}=sw_{1}s^{-1}$
for some word $s$ in $\left\{ a_{1}^{\pm1},\ldots,b_{g}^{\pm1}\right\} $. 
\item Replace $w_{1}$ by the new cyclically shortest representative of
$w_{1}^{\Gamma_{g}}$, replace $w_{2},\ldots,w_{k}$ by $w_{2}^{~s},\ldots,w_{k}^{~s}$,
and shorten the latter $k-1$ words using Dehn's algorithm (namely,
consider the path representing $w_{i}$ and repeatedly replace long
blocks or long chains with their complements). Rename the new words
$w_{1},\ldots,w_{k}$, and replace $J$ with $J^{s}$ (recall that
this does not change the corresponding core surface).
\item Now construct a wedge of $k$ petals, where petal $i$ consists of
$\left|w_{i}\right|$ directed edges labeled by $\left\{ a_{1},\ldots,b_{g}\right\} $
so that it reads the word $w_{i}$. Call the resulting directed and
edge-labeled graph $Y_{1}$.
\end{itemize}
\item \textbf{Folding and boundary reduction: }Let $Y=Y_{1}.$ Perform the
following two steps \uline{alternately} until none of them is possible,
always beginning with folding:
\begin{itemize}
\item \textbf{Folding edges and $4g$-gons: }Fold the $1$-skeleton of $Y$
(in the sense of Stallings: so repeatedly merge together pairs of
equally-labeled edges with the same head or the same tail), and remove
multiplicities of $4g$-gons sharing the same oriented boundary, so
that there is at most one $4g$-gon attached to any closed $\left[a_{1},b_{1}\right]\ldots\left[a_{g},b_{g}\right]$
path.
\item \textbf{Boundary reduction: }If $\partial Y$ contains a long block,
choose one such block and add a new $4g$-gon along it so that this
block in $\partial Y$ is replaced by its complement. Otherwise, if
$\partial Y$ contains a long chain, choose one such long chain and
add $4g$-gons along it so that this long chain in $\partial Y$ is
replaced by its complement. 
\end{itemize}
Call the resulting complex $Y_{2}$.
\item \textbf{Strong boundary reduction: }Finally, as long as $\partial Y$
contains a half-block or a half-chain, add new $4g$-gon along them
so that this piece of $\partial Y$ is replaced by the complement
of the half-block or half-chain. Call the resulting complex $Y_{3}$.
\end{enumerate}
\end{thm}

\begin{proof}
We need to show that the process described is well-defined, i.e.,
that the boundary $\partial Y$ of $Y$ is well defined whenever we
use it (notice that $Y$ may not even be a tiled surface along the
way), that it terminates after finitely many steps, and that the resulting
complex is indeed $\core\left(J\right)$ (and so, in particular, independent
of the choices made along the way). We analyze the three parts of
the process one by one. Let $\left(\Upsilon,q\right)=J\backslash\left(\tsg,u\right)$
and let $\pi\colon\left(\Upsilon,q\right)\to\left(\Sigma_{g},o\right)$
be the covering map.

In the first part, denote by $v$ the wedge point of the $k$ petals
of $Y_{1}$. Trivially, there is a map $p\colon\left(Y_{1},v\right)\to\left(\Sigma_{g},o\right)$
and $p_{*}\left(\pi_{1}\left(Y_{1},v\right)\right)=J\le\Gamma_{g}$.
As in (\ref{eq:lift}), there is a (unique) lift $r\colon\left(Y_{1},v\right)\to\left(\Upsilon,q\right)$
with $\pi\circ r=p$.
\[
\xymatrix{ & \left(\Upsilon,q\right)\ar@{->>}[d]^{\pi}\\
\left(Y_{1},v\right)\ar[r]_{p}\ar@{-->}[ur]^{\exists!r} & \left(\Sigma_{g},o\right)
}
\]
By definition, $\core\left(J\right)$ is a subcomplex of $\Upsilon$.
Because $w_{1}$ is a cyclically shortest cycle, so is its $r$-image
in $\Upsilon$ (by Theorem \ref{thm:birman-series}), and therefore
its image is contained in $\core\left(J\right)$. In particular, $q\in\core\left(J\right)$.
Now for $i=2,\ldots,k$, the element $\gamma_{i}\in J$ represented
by $w_{i}$ has a representative in $\pi_{1}\left(\core\left(J\right),q\right)$.
Because $\core\left(J\right)$ is boundary reduced, $\core\left(J\right)^{\left(1\right)}$
contains also a shortest representative of $\gamma_{i}$ based at
$q$ (because one can perform Dehn's algorithm inside $\core\left(J\right)$).
Call this path $p_{i}$. So now $r\left(w_{i}\right)$ and $p_{i}$
are two closed paths at $\Upsilon^{\left(1\right)}$, based at $q$,
representing the same element. They lift to two paths starting at
$u$ with the same endpoint in $\tsg$. By \cite[Thm 2.8]{BirmanSeries},
any two shortest paths with the same endpoints in $\tsg^{\left(1\right)}$
differ by a finite sequence of half-block switches. This sequence
of half-block switches descends to a sequence of half-block switches
in $\Upsilon$ which turns $p_{i}$ into $r\left(w_{i}\right)$. Because
$\core\left(J\right)$ is strongly boundary reduced (by Proposition
\ref{prop:properties of core surfaces}(\ref{enu:core surfaces are strongly boundary reduced})),
these half-block switches all take place inside $\core\left(J\right)$.
We conclude that $r\left(Y_{1}\right)\subseteq\core\left(J\right)$.\\

Now consider the ``folding and boundary reduction'' part of the
process. In every folding step (of an edge or of removing $4g$-gons),
the total number of cells in $Y$ decreases, so every iteration of
``folding'' must terminate. At the end of such an iteration, properties
\textbf{P1 }and \textbf{P3 }from Proposition \ref{prop:comb-def-of-tiled-surface}
hold. As mentioned in the paragraph preceding Proposition \ref{prop:BR core graph with octagons is a tiled surface},
this guarantees that $Y$ admits a well-defined thick version and
therefore that $\partial Y$ is well defined. Thus every ``boundary
reduction'' iteration is well defined. Clearly, in every non-empty
iteration of boundary reduction, the length of $\partial Y$ strictly
decreases, and in every folding iteration, this length does not increase.
Therefore, the second part of the process is well defined and finite.

In addition, along the second part of the process, the map $p\colon\left(Y,v\right)\to\left(\Sigma_{g},o\right)$
is defined and $p_{*}\left(\pi_{1}\left(Y,v\right)\right)=J$ at every
step, so there is a corresponding lift $r\colon\left(Y,v\right)\to\left(\Upsilon,q\right)$
at every step. A folding step does not alter the image of $r$. Moreover,
every $4g$-gon added to $Y$ in a boundary reduction step, is added
along some long block (chain) at $\partial Y$, which is mapped by
$r$ to a long block (chain, respectively) in $\Upsilon$. But $\core\left(J\right)$
is boundary reduced, so for every long block (chain) in $\core\left(J\right)$,
the $4g$-gons along it also belong to $\core\left(J\right)$. By
induction we thus see that $r\left(Y\right)\subseteq\core\left(J\right)$
throughout part 2 of the process. 

At the end of the second part, $Y_{2}$ is a complex with edges directed
and labeled by $\left\{ a_{1},\ldots,b_{g}\right\} $, which satisfies
\textbf{P1} and \textbf{P3 }and which is also boundary reduced. By
Proposition \ref{prop:BR core graph with octagons is a tiled surface},
it is a tiled surface, and as in the proof of Proposition \ref{prop:comb-def-of-tiled-surface},
the unique lift $r\colon\left(Y_{2},v\right)\to\left(\core\left(J\right),q\right)$
must be an embedding. So $Y_{2}$ is a boundary reduced subcomplex
of $\core\left(J\right)$ with the same fundamental group.\\

Finally, we can think of the third part of the process as taking place
in $\Upsilon$: as in the proof of Lemma \ref{lem:SBR =00003D=00003D> every shortest representative of a cycle is contained},
it is a finite process in which the length of the boundary does not
change. As $\core\left(J\right)$ is strongly boundary reduced, the
third part never leaves $\core\left(J\right)$. So at the end of the
third part, $Y_{3}$ is a strongly boundary reduced subcomplex of
$\core\left(J\right)$ with the same fundamental group. By Lemma \ref{lem:SBR containing J contains core(J)},
we actually have $Y_{3}=\core\left(J\right)$. 
\end{proof}

\section{Epilogue}

This paper came out as a side of our work \cite{magee2020asymptotic,magee2022random}
on random coverings of compact surfaces. We tried to elaborate here
on some basic properties of core surfaces which we use in ibid, as
well as some basic properties that illustrate the resemblance of core
surfaces to Stallings core graphs. However, we made no systematic
attempt to (re-)prove results about subgroups of the surface group
$\Gamma_{g}$ using core surfaces. We believe core surfaces should
be useful here, and think that a more systematic attempt in this direction
should be taken in the future.

\bibliographystyle{alpha}
\bibliography{core_surfaces}
Michael Magee, Department of Mathematical Sciences, Durham University,
Lower Mountjoy, DH1 3LE Durham, United Kingdom

\noindent \texttt{michael.r.magee@durham.ac.uk}\\

\noindent Doron Puder, School of Mathematical Sciences, Tel Aviv University,
Tel Aviv, 6997801, Israel\\
\texttt{doronpuder@gmail.com}
\end{document}